\documentclass[11pt, oneside, tikz, border=10pt]{amsart}
\usepackage[utf8]{inputenc}

\usepackage{verbatim}
\usepackage{amsmath,amsthm,verbatim,amssymb,amsfonts,amscd,graphicx,calc,geometry,mathtools}

\usepackage{xcolor,kotex}
\usepackage{graphics}
\usepackage{float}

\usepackage{euscript, enumerate,hhline,array}
\usepackage{url,hyperref}

\newcommand{\p}{\partial}

\newcommand{\be}{\begin{equation}}
\newcommand{\ba}{\begin{aligned}}
\newcommand{\bee}{\begin{equation*}}
\newcommand{\ee}{\end{equation}}
\newcommand{\ea}{\end{aligned}}
\newcommand{\eee}{\end{equation*}}
\newcommand{\bea}{\begin{equation} \begin{aligned} }
\newcommand{\eea}{\end{aligned}\end{equation} }

\newcommand{\R}{\mathbb{R}}

\theoremstyle{plain}
\newtheorem{theorem}{Theorem}[section]

\newtheorem{corollary}[theorem]{Corollary}

\newtheorem{lemma}[theorem]{Lemma}
\newtheorem{proposition}[theorem]{Proposition}

\theoremstyle{remark}

\newtheorem{remark}[theorem]{Remark}

\theoremstyle{definition}
\newtheorem{definition}[theorem]{Definition}

\numberwithin{equation}{section}

\title{Ancient Gauss curvature flows of bounded width}

\author{Beomjun Choi}
\address{BC: Department of Mathematical Sciences, KAIST, 291 Daehak-ro, Yuseong-gu, Daejeon 34141, Republic of Korea}
\email{bchoi@kaist.ac.kr}

\author{Kyeongsu Choi}
\address{KC: Korea Institute for Advanced Study, 85 Hoegi-ro, Dongdaemun-gu, Seoul 02455, Republic of Korea}
\email{choiks@kias.re.kr}

\author{Dongjun Noh}
\address{DN: Department of Mathematics, POSTECH, 77 Cheongam-ro, Nam-gu, Pohang, Gyeongbuk 37673, Republic of Korea}
\email{ndj0727@postech.ac.kr}

\begin{document}

\begin{abstract}
In this paper, we construct a pancake-like ancient compact solution with flat sides to the Gauss curvature flow, contained in a slab. Also, we construct sausage-like ancient  compact solutions to the $\alpha$-Gauss curvature flow with $\alpha >\frac{1}{2}$, asymptotic to a round cylinder.
\end{abstract}

\maketitle

\tableofcontents

\section{Introduction}\label{sec 1}

In 1974 \cite{firey1974shapes}, Firey introduced the Gauss curvature flow as a model for the wear of stones by tidal waves. More precisely, since the Gauss curvature describes the infinitesimal area distortion under the Gauss map, he considered the evolution of convex surfaces $\Sigma_t \subset \mathbb{R}^3$ shrinking by their Gauss curvature. This is a natural geometric parabolic Monge-Ampere equation, since a solution to the flow is a one-parameter family of convex hypersurfaces $\Sigma_t\subset \mathbb{R}^{n+1}$ with embeddings $X:M^n\times [0,T)\to \mathbb{R}^{n+1}$ such that $\Sigma_t=X(M^n,t)$ and
\begin{equation}
    X_t=-K \nu,
\end{equation}
where $K$ and $\nu$ denote the Gauss-Kronecker curvature and the outward-pointing unit normal vector of $\Sigma_t$, respectively.

\bigskip

In 1994, Hamilton \cite{hamilton1994remarks} pointed out a striking difference between the Gauss curvature flow and the mean curvature flow, namely that flat sides persist for a positive time under the Gauss curvature flow. If the initial hypersurface is weakly convex, then the mean curvature flow immediately makes it strictly convex. In contrast, Hamilton constructed barriers for the Gauss curvature flow with flat sides, demonstrating their persistence for a positive time under the evolution. Later, Chopp-Evans-Ishii \cite{chopp1999waiting} shows that if a flow with flat sides is of class $C^3$ then its flat sides do not move at all for positive time.

\bigskip

The persistence of flat sides of the Gauss curvature flow is closely related to the existence of compactly supported solutions $u\geq 0$ to the porous medium equation with high exponent $m>1$
\begin{equation}
    u_t=\Delta u^m,
\end{equation}
which is the so-called slow diffusion. Indeed, analogous to the waiting flat sides in the smooth Gauss curvature flow, if the initial data $u(\cdot,0)$ to the porous medium equation is compactly supported and smooth, then its support does not move for positive time. See \cite{aronson1970regularity, knerr1977porous}. On the other hand, the compact support moves once its pressure $u^{m-1}$ becomes non-degenerate as shown by Daskalopoulos and Hamilton \cite{daskalopoulos1998regularity}. They observed in \cite{daskalopoulos1999free} that the Gauss curvature flow with flat side also has a function which plays the role of the pressure in the porous medium equation. More precisely, near a flat side, we consider the flow as the graph of a convex non-negative function $u$ with flat side $\{(x,0) \in \mathbb{R}^3: u(x,t)=0\}$. Then, $u$ solves
\begin{equation}
    u_t=\frac{\det D^2u}{(1+|Du|^2)^{\frac{3}{2}}},
\end{equation}
and its square root $w:=\sqrt{u}$ plays the role of the pressure function. Hence, if the surface consists of a flat side and a strictly convex part and the pressure enjoys the optimal regularity and non-degeneracy
\begin{equation}\label{eq:nondegenerate}
    C^{-1} \leq |Dw| \leq C,
\end{equation}
then the flat side shrinks with positive speed for a short time, preserving the condition \eqref{eq:nondegenerate}. Later, Daskalopoulos and K.A.Lee \cite{daskalopoulos2004worn} proved the long time existence result that the flat side shrinks with positive speed until it converges to a point. See also the higher dimensional analogue \cite{HWZ2024GCF}, and further related works \cite{savin2005obstacle,HTW24MAfree,DL12MAfree,KLR013aGCF}. Note that Andrews \cite{andrews1999gauss} showed that any compact viscosity solution to the Gauss curvature flow in $\mathbb{R}^3$ becomes of class $C^{1,1}$ immediately, namely it enjoys $C^{1,1}$-regularity for all positive time. Thus, due to the persistence of \eqref{eq:nondegenerate} near the flat side, the $C^{1,1}$-regularity is optimal. Furthermore, Andrews \cite{andrews1999gauss} showed that any compact Gauss curvature flow in $\mathbb{R}^3$ becomes strictly convex in finite time, and then it converges to a round point. Therefore, any possible flat side disappears before the flow develops a singularity.

\bigskip
\bigskip

In many works, the flow was generalized to other dimensions and to powers of the Gauss-Kronecker curvature, giving rise to the $\alpha$-\textit{Gauss curvature flow} ($\alpha$-GCF). Given a fixed $\alpha>0$, the $\alpha$-Gauss curvature flow in $\mathbb{R}^{n+1}$ is a one-parameter family of convex hypersurfaces $\Sigma_t\subset \mathbb{R}^{n+1}$ with embeddings $X:M^n\times [0,T)\to \mathbb{R}^{n+1}$ such that $\Sigma_t=X(M^n,t)$ and
\begin{equation}
    X_t=-K^\alpha \nu,
\end{equation}
where $K$ and $\nu$ denote the Gauss-Kronecker curvature and the outward-pointing unit normal vector of $\Sigma_t$, respectively. In particular, we refer to the case $\alpha = 1, n\geq 2$ as the Gauss curvature flow, to the case $\alpha=1,n=1$ as the \textit{curve shortening flow}, to the case $\alpha >0, n=1$ as the $\alpha$-\textit{curve shortening flow}, and to the case $\alpha = \frac{1}{n+2}, n \geq 1 $ as the \textit{affine normal flow}.

\bigskip

We call $\alpha=\frac{1}{n+2}$ the affine critical power, since a solution to the affine normal flow remains a solution under volume-preserving affine transformations of the ambient space $\mathbb{R}^{n+1}$. Therefore, any ellipsoid is a self-shrinker to the affine normal flow. Indeed, Andrews \cite{andrews1996contraction} showed that any compact solution to the flow converges to an ellipsoid after rescaling. In contrast, for the super-affine-critical power $\alpha>\frac{1}{n+2}$, the convergence of a compact solution in $\mathbb{R}^{n+1}$ to a round sphere after rescaling is proven by many authors. See the direct proofs of convergence to a round sphere for $\alpha =\frac{1}{n}, n \geq 2$ by Chow \cite{chow1985deforming}, for $\alpha=1, n=2$ by Andrews \cite{andrews1999gauss}, and for $\alpha \in [\frac{1}{2},1], n=2$ by Andrews-Chen \cite{andrews2012surfaces}. For the general case, the result can be proven by combining convergence to shrinkers with classification of shrinkers. See the convergence to a shrinker for $\alpha \in [\frac{1}{n+2},\frac{1}{n}]$ by Andrews \cite{andrews2000motion}, for $\alpha=1$ by Guan-Ni \cite{guan2017entropy}, and for $\alpha \geq \frac{1}{n+2}$ for Andrews-Guan-Ni \cite{andrews2016flow}. The classification of self-similar solutions was obtained by the second author for $\alpha\in [\frac{1}{n},1+\frac{1}{n})$ in his thesis \cite{choi2017gauss}, and for $\alpha \geq \frac{1}{n+2}$ by the second author together with Brendle-Daskalopoulos \cite{brendle2017asymptotic}. See also the convergence result in $\mathbb{R}^2$ for $\alpha>\frac{1}{3}$ by Andrews \cite{andrews1998evolving}.

\bigskip

On the other hand, for the sub-affine-critical power $\alpha \in (0,\frac{1}{n+2})$, a compact solution to the $\alpha$-GCF ($\alpha$-CSF in $\mathbb{R}^2$) would generically develop a Type II singularity. For example, see \cite{andrews2002non} for $n=1$. This is because the $\alpha$-flow has an entropy \cite{andrews2016flow} which has a universal lower bound for $\alpha \geq \frac{1}{n+2}$ and a universal upper bound for $\alpha \leq \frac{1}{n+2}$; see also \cite{choi2025classification}. Hence, for the small power $\alpha \in (0,\frac{1}{n+2}]$, we can expect the classification result for compact \textit{ancient flows}, which are solutions that exist from time $-\infty$. See the classification results for shrinkers in $\mathbb{R}^2$ by Andrews \cite{andrews2003classification}, for compact ancient flows with $\alpha<\frac{1}{3},n=1$ by the second author together with Sun \cite{choi2022ancient,choi2025classification}, for compact ancient affine normal flows in $\mathbb{R}^2$ by Chen \cite{chen2015classifying} and Ivaki \cite{ivaki2016classification}, and for compact ancient affine normal flows with $n \geq 2$ by Loftin-Tsui \cite{loftin2008ancient}. Note that these compact ancient flows with $\alpha \in (0,\frac{1}{n+2}]$ converge to shrinkers after rescaling as time goes negative infinite, namely they are Type I ancient flows. This convergence is easily expected due to the universal entropy upper bound for compact strictly convex bodies. We also note that for sub-affine-critical powers, non-compact singularity models such as translators may appear in Type II singularities. A series of recent works by the first and second authors together with S. Kim \cite{CCK_translating, CCK_existence, CHOI2025113849} classified all such translators for $n=2$ and $\alpha \in (0,\tfrac{1}{4})$. These translators are entire smooth graphical solutions whose level curves are modeled on Andrews’ shrinkers \cite{andrews2003classification}.

\bigskip

In contrast, for $\alpha >\frac{1}{n+2}$ when there is no universal upper bound of the entropy, there are Type II compact ancient flows by power $\alpha$ of the Gauss curvature. For example, the \textit{paperclip} is a well-known compact Type II ancient curve shortening flow, which is obtained by gluing two \textit{grim reaper} curves, the unique translator. Indeed, Daskalopoulos-Hamilton-Sesum \cite{daskalopoulos_classification_2010} proved that a convex compact ancient curve shortening flow is either a shrinking circle or a paperclip. See also the uniqueness of Type I compact ancient CSFs by X.J. Wang \cite{wang2011convex}, and the classification of complete ancient CSFs by Bourni-Langford-Tinaglia \cite{bourni2020convex}. For $\alpha>\frac{1}{2},n=1$ by solving a second order ODE one can easily observe that the translator is contained in a slab. In particular, if $\alpha> 1,n=1$ then the translator has the flat sides; see the proof of Lemma \ref{alpha > 1 flat side}. Indeed, for all $n\geq 1$, given any convex bounded open domain $\Omega\subset \mathbb{R}^n$ there is a unique translator with $\alpha>\frac{1}{2}$ asymptotic to $\partial\Omega\times \mathbb{R}$ by Urbas \cite{urbas1988global,urbas1999complete}. More precisely, the translator is the graph of a strictly convex function $u \in C^\infty(\Omega)$ in $\Omega\times \mathbb{R}$ such that $|Du(x)|\to \infty$ as $x\to \partial\Omega$. Note that if $u(x)\to u(x_0) < +\infty$ as $x\to x_0 \in \partial \Omega$, then the translator includes a ray $\{(x_0,t):t \geq u(x_0)\}$, which can be a part of a flat side. Therefore, for $\alpha>1,n=1$ one can immediately obtain an ancient flow with two flat sides by placing two translating flows in opposite directions; see subsection \ref{subsec 6.2}. We call such compact ancient flows in a slab \textit{paperclips}. Bourni-Clutterbuck-X.H.Nguyen-Stancu-G.Wei-V.M.Wheeler \cite{bourni2022ancient} constructs a strictly convex paperclip for each $\alpha\in (\frac{1}{2},1)$, and they also showed that a compact ancient flow with $\alpha\in (\frac{2}{3},1)$ and $n=1$ is either a shrinking circle or a paperclip as in the case $\alpha=n=1$ \cite{daskalopoulos_classification_2010}. For higher dimensions $n\geq 2$, together with Daskalopoulos the first and second authors \cite{choi2024uniqueness} showed the existence and the uniqueness of a compact ancient Gauss curvature flow asymptotic to $\partial \Omega \times \mathbb{R}$, where $\Omega\subset \mathbb{R}^n$ is a convex bounded domain with $C^{1,1}$ boundary $\partial \Omega$. 

 \bigskip

\subsection{Main theorems}
\label{sec 1.1}

In this paper, we construct compact symmetric ancient solutions to $\alpha$-Gauss curvature flow in $\mathbb{R}^3$, contained in a slab $I\times \mathbb{R}^2$ or in a bounded cylinder $\Omega \times \mathbb{R}$.

\begin{figure}[htbp]
    \centering
    \includegraphics[width=0.8\linewidth]{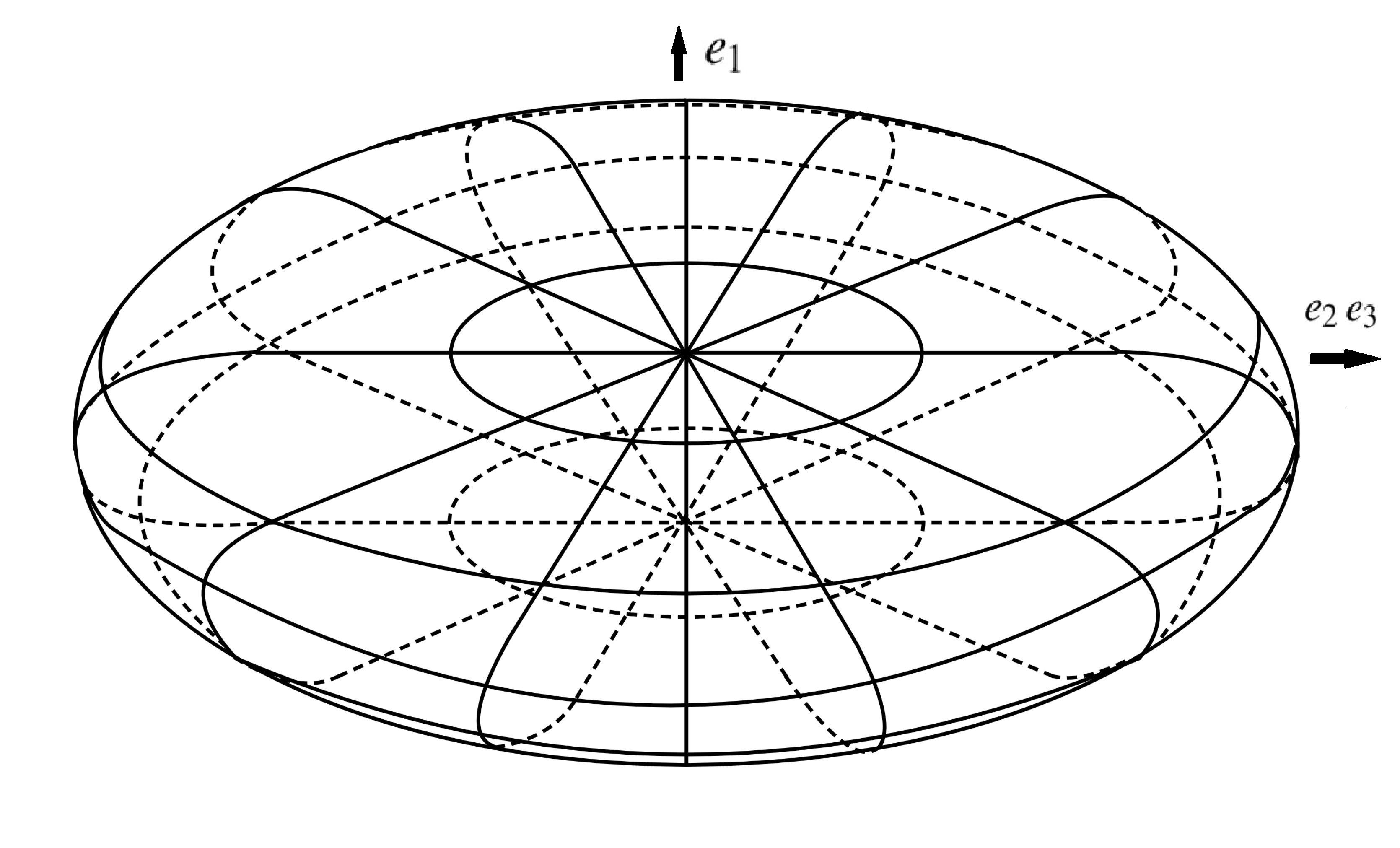}
    \caption{Ancient pancake}
    \label{fig:pancake solution}
\end{figure}

In \cite{bourni2021collapsing}, Bourni-Langford-Tinaglia constructed a so-called \textit{ancient pancake}, an $Z_2\times O(n)$-invariant ancient solution to mean curvature flow, which is asymptotic to $\{ \mathbf{x} \cdot e_1= \pm \alpha\}$ as time goes to negative infinite for some $\alpha \in (0,+\infty)$. See Figure \ref{fig:pancake solution}. Moreover, they showed that an $O(n)$-symmetric ancient flow asymptotic to the hyperplanes $\{ \mathbf{x} \cdot e_1  = \pm \alpha\}$ is unique. Similarly, we also call an ancient solution to the Gauss curvature flow by ancient pancake if it is $Z_2\times O(n)$-symmetric and asymptotic to two hyperplanes as time goes to negative infinite.

\begin{theorem}[ancient pancake with flat sides]
\label{Main theorem (slab)}
There exists an ancient pancake for the Gauss curvature flow in $\mathbb{R}^3$, namely, a compact convex viscosity solution of class $C^{1,1}$ with $Z_2\times O(2)$ symmetry, contained in the slab between the two hyperplanes $\{ \mathbf{x}\cdot e_1=\pm \alpha\}$. Moreover, the solution has flat sides, in the sense that for all sufficiently negative times it touches both hyperplanes.
\end{theorem}

The ancient pancake to the mean curvature flow is smooth and strictly convex due to the strong maximum principle. In contrast, the Gauss curvature flow allows flat sides, since its ellipticity degenerates as the principal curvatures approach zero. The existence of ancient Gauss curvature flows with flat sides was first conjectured by Hamilton. He conjectured that a translator over the domain $\Omega\subset \mathbb{R}^2$ has flat sides if $\partial \Omega$ is a polygon. Later, together with K.A.Lee-Daskalopoulos, the second author \cite{choi2021translating} gave an affirmative answer: if $\partial\Omega$ contains a flat segment $I\subset \partial \Omega$, then the translator has a flat side on $I\times \mathbb{R}$ which is asymptotic to the two lines $\partial I \times \mathbb{R}$. Our theorem \ref{Main theorem (slab)} provides the first existence result of an ancient non-self-similar flow with flat sides.

\begin{figure}[htbp]
    \centering
    \includegraphics[width=0.9\linewidth]{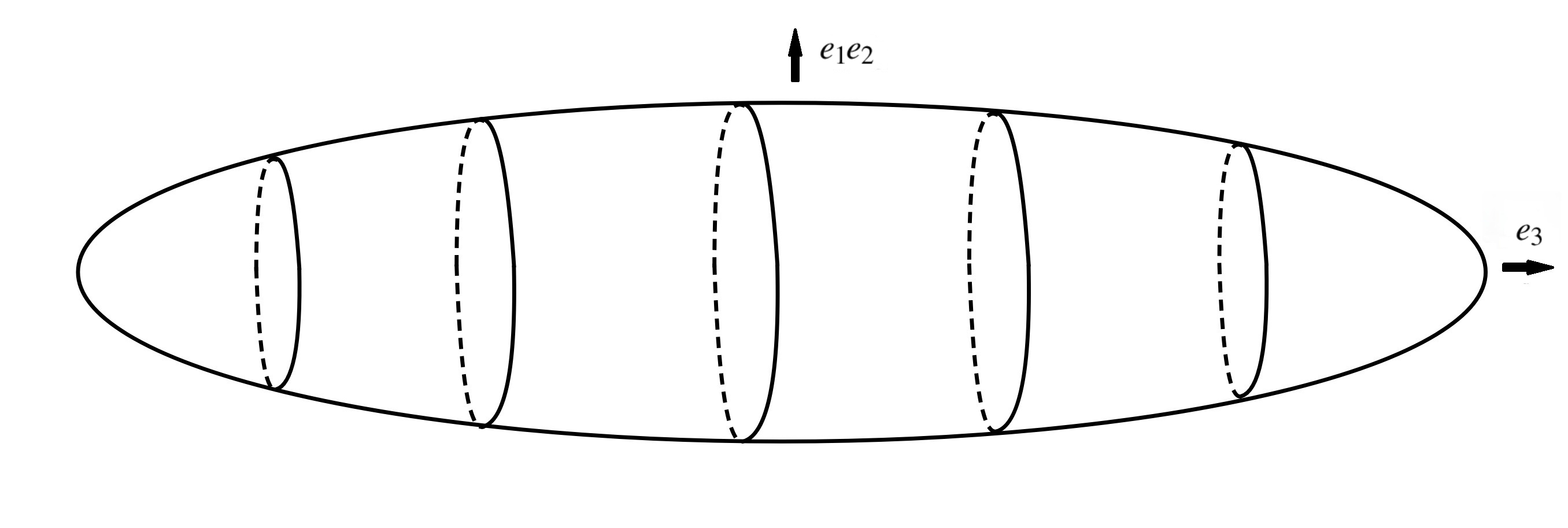}
    \caption{Ancient sausage}
    \label{fig:sausage solution}
\end{figure}

Next, we generalize the results in \cite{choi2024uniqueness,bourni2022ancient} for $n=2, \alpha >\frac{1}{2}$. More precisely, we consider \textit{ancient sausages}, which are compact ancient solutions with $O(n)\times Z_2$-symmetry, asymptotic to a cylinder $\partial B \times \mathbb{R}$ for a ball $B \subset \mathbb{R}^n$ as time goes negative infinite. See the analogous ancient solution to the two-dimensional Ricci flow discovered by King \cite{king1993exact,king1994} and Rosenau \cite{rosenau1995fast}, which can be obtained by gluing two cigar solitons in \cite{hamilton1988ricci}. See also the classification of compact ancient two-dimensional Ricci flows by Daskalopoulos-Hamilton-Sesum \cite{daskalopoulos2012classification}.

\begin{theorem}[ancient sausage]
\label{Main theorem (cylinder)}
   The $\alpha$-Gauss curvature flow in $\mathbb{R}^3$ with $\alpha>\frac{1}{2}$ has an ancient sausage, which is a viscosity solution. In particular, for $\alpha>1$ it touches the asymptotic cylinder $\partial B \times \R$ for sufficiently negative time.
\end{theorem}

\subsection{Outline of the paper}
\label{sec 1.2}
In Section 2, we recall known results which will be used in the following sections. In Section \ref{sec 3} and \ref{sec 5}, we locate the minimum points for the Gauss curvature $K$ on the rotationally symmetric surface under the flow. The minimum is important to use a touching circle technique introduced in \cite[ Lemma 4.4]{bourni2021collapsing}.  By using the minimum Gauss curvature, in Section \ref{sec 4} and \ref{sec 6}, we show that approximate flows have some uniform bounds for width so that we prove Theorem \ref{Main theorem (slab)} and \ref{Main theorem (cylinder)}, respectively.

\bigskip

\section{Preliminaries} 
\label{sec 2}

Note that ancient solutions to the $\alpha$-GCF are not a priori smooth. Indeed \cite{chopp1999waiting} implies that our ancient pancake with flat sides cannot be globally smooth. We choose to work with a weak notion of solution to the $\alpha$-Gauss curvature flow, for $\alpha>0$. 

\begin{definition}[viscosity solution, c.f. Section 8.2 \cite{andrews2000motion},  Definition 2.6 \cite{choi2022convergence}]\label{def-weaksol}
Let $\mathcal{K}_t \subset \mathbb{R}^{n+1}$, for $t \in (t_1,t_2)$, be a one-parameter family of bounded convex bodies\footnote{Here, a convex body refers to a closed convex set which has non-empty interior.}. We say $\Sigma_t =\p \mathcal{K}_t \subset \mathbb{R}^{n+1}, t \in (t_1,t_2)$,  is a {\em viscosity subsolution} to the $\alpha$-GCF if the following holds: for every $[t_1',t_2']\subset (t_1,t_2)$, if  $\Sigma'_t=\p \mathcal{K}_t' $ is a smooth strictly convex solution to the $\alpha$-GCF with $\mathcal{K}_{t_1'}' \subset \mathcal{K}_{t_1'}$, then the comparison $\mathcal{K}'_t \subset \mathcal{K}_t$ holds  for all $t\in[ t'_1,t_2']$. Similarly we define {\em viscosity supersolution} using the other inclusion. $\Sigma_t$ is a {\em viscosity solution}  if it is both a viscosity subsolution and a viscosity supersolution. 
\end{definition}

\bigskip

We have the existence and uniqueness of a viscosity solution emanating from the boundary of a bounded convex body. The proof is already given in the literature by taking a limit of approximating smooth strictly convex solutions. We omit this proof here. 

\begin{proposition}[existence of unique solution, Theorem 15 \cite{andrews2000motion} c.f. Theorem 2.7 \cite{choi2022convergence}]\label{prop:viscosity_uniqueness} For a given closed\footnote{A closed surface refers a compact surface without boundary. Since we are working with surfaces without boundary, we often mention it a compact surface if no confusion is made.} convex surface 
 $\Sigma_0\subset \mathbb{R}^3$ which bounds a convex body, there exists a unique viscosity solution $\Sigma_t$ to the $\alpha$-Gauss curvature flow with $\alpha>0$. The solution exists for $t\in(0,T)$ for some $T\in(0,\infty)$, $\Sigma_t$ converges to $\Sigma_0$ in the Hausdorff distance as $t\to 0$, and $\Sigma_t$ converges to a point as $t\to T^-$.
\end{proposition}

\bigskip

\begin{remark} \label{remark-hausdorffcont}
If $\Sigma_t$, for $t\in(0,T)$, is a compact convex viscosity solution to the $\alpha$-GCF, then $\Sigma_t$ depends continuously on $t$ in the Hausdorff distance; in particular, it cannot jump in time. Indeed, $\Sigma_t$ can be approximated by smooth strictly convex solutions, whose support functions satisfy a uniform modulus of continuity estimate in time by \cite[Theorem 5]{andrews2000motion}. Passing to the limit, one obtains the same estimate for $\Sigma_t$, and hence the desired continuity with respect to the Hausdorff distance.
\end{remark}

\bigskip

As we construct ancient solutions through limits of almost ancient solutions, a suitable compactness theorem of the following form will be needed.  

\begin{lemma}[compactness]\label{lemma-extra}
Let $\Sigma^i_t$ be a sequence of viscosity solutions to the $\alpha$-GCF with convex closed initial surfaces $\Sigma^i_0$ defined on a fixed time interval $(0,T)$ some $T\in(0,\infty)$. Suppose $\Sigma^i_0$ converges to $\Sigma_0$ in the Hausdorff distance.  Then (i) the viscosity solution to the $\alpha$-GCF emanating from $\Sigma_0$, namely $\Sigma_t$, exists at least for $t\in(0,T)$, and  (ii) $\Sigma^i_t$ converges to $\Sigma_t$, for all $t\in(0,T)$, in the Hausdorff distance. 
\end{lemma}

\begin{proof}The proof follows by the comparison principle and an approximation argument. After a translation, we may assume the origin is contained (strictly) in the inside of the region bounded by $\Sigma_0$. Consider a family of smooth strictly convex surfaces $\bar \Sigma^i_0$, $i=1,2,\ldots$, which strictly increases to $\Sigma_0 $ from inside. By Proposition \ref{prop:viscosity_uniqueness}, the smooth unique flow emanating from $\bar \Sigma^i_0$, namely $\bar \Sigma^i_t$, exists for $t\in [0,T_i)$, where $T_i\in (0,\infty)$ denotes the extinction time for $\bar \Sigma^i_t$. Let us denote the extinction time for $\Sigma_t$ by $T_*(\Sigma_0)$. In this proof, we denote the convex bodies surrounded by $\bar \Sigma^j_t$, $\Sigma^i_t$, $\Sigma_t$, by $\bar{ \mathcal{K}}^j_t$, ${\mathcal{K}}^i_t$, ${\mathcal{K}}_t$, respectively.

 By the choice of $\bar \Sigma^j_t=\partial \bar {\mathcal{K}}^j_t$, observe that, for each fixed $\varepsilon>0$, there exists $j_0=j_0(\varepsilon)\in \mathbb{N}$ such that, for all $j\ge j_0$,
\[ \mathrm{int}(\bar{  \mathcal{K}}^j_0 ) \subset \subset \mathrm{int}( \mathcal{K}_0) \subset \subset   \mathrm{int}( (1+\varepsilon) \bar{ \mathcal{K}}^j_0).\]
Similarly, for each $\varepsilon>0$ and $j\ge j_0(\varepsilon)$, there exists $i_0=i_0(\varepsilon,j)\in \mathbb{N}$ such that, for all $i\ge i_0$, 
\[ \mathrm{int}(\bar{ \mathcal{K}}^j_0)\subset  \subset \mathrm{int}( \mathcal{K}^i_0) \subset \subset \mathrm{int}((1+\varepsilon)\bar{ \mathcal{K}}^j_0).\]
Recall, by the rescaling, the flow emanating from $(1+\varepsilon)\bar \Sigma^j_t$ can be written as  $(1+\varepsilon)\bar\Sigma^j_{(1+\varepsilon)^{-(1+2\alpha)} t} $ for $t\in(0, (1+\varepsilon)^{1+2\alpha}T_j$). In view of the  comparison principle and the inclusion relations above, we conclude that (1) $T_j$ is increasing, (2) $T_j \le T_*(\Sigma_0)$, and (3) if $j\ge j_0(\varepsilon)$, then $T\le (1+\varepsilon)^{1+2\alpha} T_j $. By taking $j\to \infty $ and then $\epsilon \to 0$, we conclude $T_*(\Sigma_0)\ge \lim_{j\to \infty} T_j \ge T$. This proves the assertion (i).

Next, again by the comparison principle, for given $t\in(0,T)$, if $j\ge j_0(\varepsilon)$, $t <T_j$, and $i\ge i_0(\varepsilon,j)$, then
\begin{equation} \label{eq-965}  \bar{ \mathcal{K}}^j_t \subset \mathcal{K}_t \subset (1+\varepsilon) \bar{ \mathcal{K}}^j_{(1+\varepsilon)^{-(1+2\alpha)}t}, \quad \bar{ \mathcal{K}}^j_t \subset \mathcal{K}^i_t \subset (1+\varepsilon) \bar{ \mathcal{K}}^j_{(1+\varepsilon)^{-(1+2\alpha)}t}.\end{equation}

We first claim that $\bar \Sigma^j_t$ converges to $\Sigma_t$ as $j\to \infty$. From the definition, the increasing limit of $\alpha$-GCF, namely $\bar \Sigma_t := \lim_{j\to \infty} \bar \Sigma^j_t$ is a supersolution. To check $\bar \Sigma_t$ is a subsolution, if $ \Sigma'_t $ is any other comparing solution from the inside, one observes, by the comparison principle, that $\Sigma'_t$ is contained in $(1+\varepsilon) \bar \Sigma^j_{(1+\varepsilon)^{-(1+2\alpha)}t}$ for large $j$. Taking $j\to \infty$ and then $\varepsilon \to0$, we see $\Sigma'_t$ is contained in the inside of $\bar \Sigma_t$. Thus $\bar \Sigma_t $ is a subsolution. Since $\bar \Sigma_0=\Sigma_0$, by the uniqueness of solution, we obtain $\lim_{j\to \infty} \bar \Sigma^j_t= \bar \Sigma_t = \Sigma_t$ for all $t\in (0,T)$. Finishing the proof of  (ii), from \eqref{eq-965}, we have \[d_H(\mathcal K_t^i,\mathcal K_t)
\le
d_H\!\left(
\bar{\mathcal K}_t^j,\,
(1+\varepsilon)\bar{\mathcal K}_{(1+\varepsilon)^{-(1+2\alpha)}t}^j
\right).\]
Considering $i\to \infty$ and then $j\to\infty$, 
\[ \limsup_{i\to \infty} d_H(\mathcal K_t^i,\mathcal K_t)
\le
d_H\!\left(
{\mathcal K}_t,\,
(1+\varepsilon){\mathcal K}_{(1+\varepsilon)^{-(1+2\alpha)}t}
\right).\] Finally taking $\varepsilon \to0 $, by the continuity of $\Sigma_t $ in the Hausdorff distance (in $t$) (see Remark~\ref{remark-hausdorffcont}), we conclude the desired convergence (ii).
\end{proof}

  \bigskip
 
\begin{definition}[displacement]\label{displacements}
We denote the horizontal and vertical displacements of a $Z_2^2$-symmetric curve $\Gamma \in \mathbb{R}^2$ by $h$ and $l$, respectively. Namely,
\begin{align}
&  h := \sup \{\mathbf{x}^1:\mathbf{x}\in \Gamma\}, && l :=\sup \{ \mathbf{x}^2:\mathbf{x} \in \Gamma\}.
\end{align}
Also, given a $Z_2\times O(2)$ or $O(2)\times Z_2$-symmetric surface $\Sigma\subset \mathbb{R}^3$, we denote
\begin{align}
&  h := \sup \{\mathbf{x}^1:\mathbf{x}\in \Sigma\}, && l :=\sup \{ \mathbf{x}^3:\mathbf{x} \in \Sigma\}.
\end{align}
\end{definition}

 \bigskip

\begin{proposition}[c.f. {\cite[Lemma 4.4]{bourni2021collapsing}}] \label{prop-touching}
Let $\Gamma\subset\mathbb{R}^2$ be a smooth convex $Z_2^2$-symmetric closed curve whose displacements in Definition \ref{displacements} satisfy $l \geq h$. Then, given $z\in [h,l]$ there exists $\mathbf{x}_z\in \Gamma$ such that $|\langle \mathbf{x}_z,e_2\rangle| \leq z$ and the curvature $\kappa>0$ of $\Gamma$ at $\mathbf{x}_z$ is bounded by
\begin{equation}
    \kappa(\mathbf{x}_z) \leq \frac{2h}{z^2+h^2}.
\end{equation}
\end{proposition}

\begin{proof}
We can obtain the desired result by a trivial modification of the proof of \cite[Lemma 4.4]{bourni2021collapsing}. However, we briefly explain it for readers' convenience. 

We consider a circle centered at $-(\rho-h)e_1$ with radius $\rho:=\frac{z^2+h^2}{2h}$. Then, the points $\pm ze_2$ and $he_1$ belong to the circle. We denote by $\mathcal{C}$ the arc of the circle containing $he_1$ and $\partial \mathcal{C}=\{\pm ze_2\}$. 

If $\kappa(he_1) \leq 1/\rho$, then we can complete the proof by choosing $\mathbf{x}_z=he_1$. Otherwise, a portion of $\mathcal{C}$ near the point $he_1$ lies outside of $\Gamma$. Hence, we shift $\mathcal{C}$ down to make it tangent to $\Gamma$. Namely, there is $a\in (0,h)$ such that $\mathcal{C}-ae_1$ is tangent to $\Gamma$. Then, by choosing one of the tangent points as $\mathbf{x}_z$, we finish the proof.
\end{proof}

\bigskip

We recall the \textit{paperclip} $\overline{\Gamma}_t=\bar \gamma(S^1,t)$ that converges to $\{x_1=\pm \frac{1}{2}\pi\}$ as $t\to -\infty$ and shrinks to the origin as $t\to 0$. Then, its position vector $\bar \gamma=(\bar\gamma^1,\bar\gamma^2)$ satisfies (e.g. equation (3) of \cite{bourni2021collapsing})
\begin{equation}\label{paperclip identity}
\cos \bar\gamma^1 = e^t \cosh{\bar\gamma^2}.
\end{equation}
\begin{figure}[htbp]
    \centering
    \includegraphics[width=0.8\linewidth]{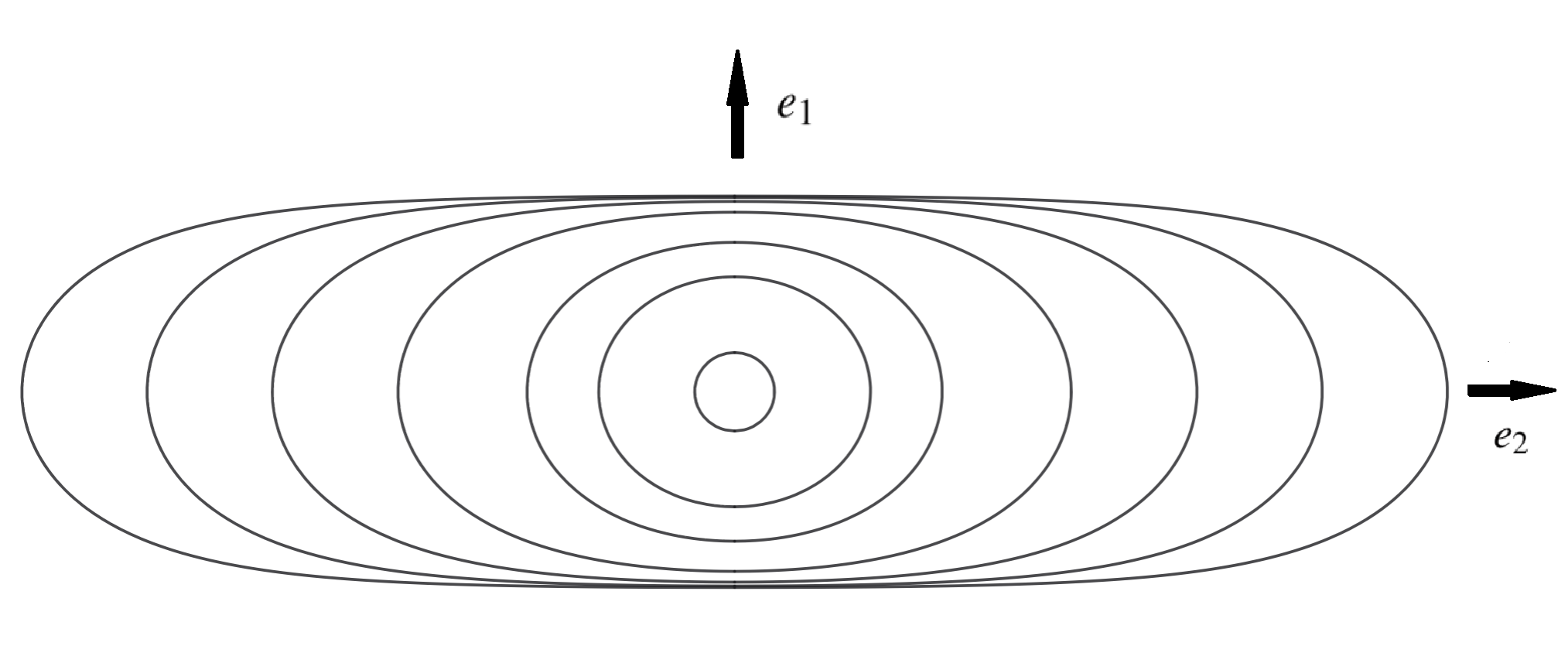}
    \caption{Time slices of paperclip}
    \label{fig:paperclip}
\end{figure}

To describe the geometry of a convex body, it is useful to introduce its support function.

\begin{definition}[Support function]
For a convex body $M^n \subset \mathbb{R}^{n+1}$, the \textit{support function} $S:S^n\to\mathbb{R}$ is defined by
\[
S(\theta):=\sup_{X\in M^n}\langle X,\theta\rangle .
\]
\end{definition}

Note that the supremum is attained at some point $X\in \partial M^n$ for which $\theta\in S^n$ is an outward unit normal vector, in the subdifferential sense. If $\partial M^n$ is smooth and strictly convex, then
\[
S(\theta)=\langle X(\theta),\theta\rangle,
\]
where $X(\theta)$ denotes the parametrization of $\partial M^n$ by the Gauss map.

Roughly speaking, the support function gives the signed distance from the origin to the supporting hyperplane whose outward normal is $\theta$.

\bigskip

\section{Curvature ratio evolution}\label{sec 3}
 
In this section, we consider a strictly convex $Z_2\times O(2)$-symmetric solution $\Sigma_t$ to the Gauss curvature flow, and we will show that the Gauss curvature attains its minimum at the axis of rotation under certain initial assumptions. To this end, we first define its \textit{profile function} $u(x,t)\geq 0$ as the unique nonnegative function such that, for all $t$, $u(\cdot,t) : [-h(t),h(t)] \to [0,\infty)$, describes the surface $\Sigma_t$ by  
\begin{equation}\label{def:pancake_profile}
\Sigma_t = \{(x,u(x,t)\cos\theta,u(x,t)\sin\theta) : |x| \leq h(t), \theta \in S^1\}.
\end{equation}
Here $h(t)$ denotes the $e_1$-displacement in Definition \ref{displacements}.

To describe its geometry, we note that
\begin{align}\label{eq:profile_symmetry}
    & u(\pm h(t),t)=0, && u_{xx}<0, && u(x,t)=u(-x,t).
\end{align}
We express the principal curvatures and the evolution equation of $u$ as follows.
 \bigskip

\begin{proposition}\label{prop:profile_evol}
   The strictly convex Gauss curvature flow $\Sigma_t$ in \eqref{def:pancake_profile} has two positive principal curvatures $\lambda_1,\lambda_2>0$ which satisfy $\lambda_1=\lambda_2$ at $|x|=h(t)$ and
    \begin{align}\label{eq:pancake_curvature}
&    \lambda_1=-\frac{u_{xx}}{(1+|u_x|^2)^{3/2}}, && \lambda_2 = \frac{1}{u(1+|u_x|^2)^{1/2}},
    \end{align}
    on  $\{|x|<  h(t)\}$. Also, the profile $u(x,t)$ solves the following parabolic equation
    \begin{equation}\label{eq:profile_equation}
        u_t = \frac{u_{xx}}{u(1+|u_x|^2)^{3/2}}
        \end{equation}
        on  $\{|x|<  h(t)\}$.
\end{proposition}

 \begin{proof}
 Thanks to the symmetry, we have $\lambda_1=\lambda_2$ at $|x|=h(t)$.
 
In the remaining of the proof, we assume $|x| <h(t)$. By using the parameters $(x,\theta)$ in \eqref{def:pancake_profile}, we denote the position vector $X(x,\theta)=(x,u\cos\theta,u\sin\theta)$. Then, we have two orthogonal tangent vectors 
\begin{align}
&    X_x=(1,u_x\cos\theta,u_x\sin\theta), && X_\theta=(0,-u\sin\theta,u\cos\theta).
\end{align}
Thus, the outward pointing unit normal $\nu$ is given by
\begin{equation}\label{eq:pancake_normal}
    \nu=\frac{(-u_x,\cos\theta,\sin\theta)}{(1+|u_x|^2)^{1/2}}.
\end{equation}
Therefore, we get the second fundamental form $h_{ij}=-\langle X_{ij},\nu\rangle$ that
\begin{align}
&    h_{xx}=-u_{xx}(1+|u_x|^2)^{-1/2}, && h_{x\theta}=0, && h_{\theta\theta}=u(1+|u_x|^2)^{-1/2}.
\end{align}
Observing $g^{xx}=(1+|u_x|^2)^{-1}$, $g^{x\theta}=0$, and $g^{\theta\theta}=u^{-2}$, we can obtain \eqref{eq:pancake_curvature}.

Next, using $X_t=(0,u_t\cos\theta,u_t\sin\theta)$ and \eqref{eq:pancake_normal}, we have $\langle X_t,\nu\rangle=u_t(1+|u_x|^2)^{-1/2}$. Since the Gauss curvature flow implies $\langle X_t,\nu\rangle=-K=-\lambda_1\lambda_2$, combining with \eqref{eq:pancake_curvature} yields \eqref{eq:profile_equation}.
\end{proof}

\bigskip

Now, we will derive the evolution equation of the curvature ratio 
\begin{equation}
R := \frac{\lambda_1}{\lambda_2} = -\frac{uu_{xx}}{1+|u_x|^2}.    
\end{equation}
For simplicity of notation, we define
\begin{equation}
    Q := (1+|u_x|^2)^{1/2}.
\end{equation}

\bigskip

\begin{lemma}\label{lem:ratio_evolution} The curvature ratio $R:=\lambda_1/\lambda_2$ of the flow $\Sigma_t$ in \eqref{def:pancake_profile} satisfies
\begin{equation}
    R_t = \frac{R_{xx}}{uQ^3} + \frac{u_x(5R-4)}{u^2Q^3}R_x + \frac{R(1-R)}{u^3Q^3}\left[R + 2|u_x|^2 (3-R)\right],
\end{equation}
for $|x| < h(t)$.
\end{lemma}

\begin{proof}
Since \eqref{eq:profile_equation} implies $R=-u^2u_tQ$, differentiating $R=-uu_{xx}Q^{-2}$ in time yields
    \begin{align}\label{eq:ratio_step1}
            R_t = -\frac{u_tu_{xx}+uu_{xxt}}{Q^2} + \frac{2u_{xx}u}{Q^3}Q_t =- \frac{u}{Q^2}u_{xxt} - \frac{2R}{Q}Q_t -\frac{R^2}{u^3Q}.
        \end{align}
Since $Q_t=u_xu_{xt}Q^{-1}$, by differentiating \eqref{eq:profile_equation} we get
\begin{equation}
    -Q_t =-\frac{u_x}{Q}\left(\frac{u_{xx}}{uQ^3}\right)_x = \frac{u_x}{Q}\left(\frac{R}{u^2Q}\right)_x = \frac{u_x}{u^2Q^2}R_x + \frac{u_xR}{Q}\left(\frac{1}{u^2Q}\right)_x.
\end{equation}
Similarly, using \eqref{eq:profile_equation} we have
\begin{equation}
    -u_{xxt} = \left(\frac{R}{u^2Q}\right)_{xx} = \frac{R_{xx}}{u^2Q} + 2R_x\left(\frac{1}{u^2Q}\right)_x + R\left(\frac{1}{u^2Q}\right)_{xx}.
\end{equation}
Combining these two equations with \eqref{eq:ratio_step1} yields
    \begin{align}
                R_t =& \,\frac{R_{xx}}{uQ^3} + R_x \left[\frac{2u}{Q^2}\left(\frac{1}{u^2Q}\right)_x + \frac{2u_xR}{u^2Q^3}\right] \\
            & +\frac{uR}{Q^2}\left(\frac{1}{u^2Q}\right)_{xx}  + \frac{2u_xR^2}{Q^2}\left(\frac{1}{u^2Q}\right)_x -\frac{R^2}{u^3Q} .
    \end{align}
To simplify the equation above, observe that
\begin{equation}\label{eq:density_derivative}
    Q_x = \frac{u_xu_{xx}}{Q} = \frac{u_x}{Q}\left(-\frac{Q^2R}{u}\right) = -\frac{u_xQR}{u}.
\end{equation}
Using \eqref{eq:density_derivative}, we get
\begin{equation}
    \left(\frac{1}{u^2Q}\right)_x = - \frac{2u_x}{u^3Q} + \frac{u_xR}{u^3Q}=(R-2)\frac{u_x}{u^3Q}.
\end{equation}
Thus, remembering \eqref{eq:density_derivative}, we differentiate again so that we get
\begin{equation}
    \left(\frac{1}{u^2Q}\right)_{xx} =R_x\frac{u_x}{u^3Q}+(R-2)\left[\frac{u_{xx}}{u^3Q}-\frac{3|u_x|^2}{u^4Q}+\frac{|u_x|^2R}{u^4Q}\right].
\end{equation}
Since $u_{xx}=-u^{-1}(1+|u_x|^2)R$, we can obtain
\begin{equation}            
    \frac{uR}{Q^2}\left(\frac{1}{u^2Q}\right)_{xx} 
            = R_x \frac{u_xR}{u^2Q^3} + \frac{(R-2)R}{u^3Q^3}\left[-R-3|u_x|^2\right].
\end{equation}
Hence, we finally get
    \begin{equation}
        R_t  = \frac{R_{xx}}{uQ^3} + \frac{R_x}{u^2Q^3}I + \frac{R}{u^3Q^3}J,
    \end{equation}
    where
    \begin{equation}
        I:=2(R-2)u_x+ 2u_xR + u_xR=(5R-4)u_x,
    \end{equation}
    and
    \begin{equation}
        J:=- (R-2)(R+ 3|u_x|^2)+2(R-2)R|u_x|^2-Q^2R.
    \end{equation}
    By using $Q=1+|u_x|^2$, we can simplify $J$ as
    \begin{equation}
        J=-R^2+R+|u_x|^2(2R^2-8R+6)=R(1-R)+2|u_x|^2(R-3)(R-1).
    \end{equation}
    This completes the proof.
\end{proof}

\bigskip

Next, by applying the maximum principle for $R$, we establish the main result in this section that if the ratio $R$ attains its minimum on the axis $e_1$ at $t=0$, then $R$ attains its interior minimum for all $t>0$.

\begin{theorem}\label{thm:ratio_min}
Suppose that  the flow $\Sigma_t$ in \eqref{def:pancake_profile} has the initial surface $\Sigma_0$ satisfying $\lambda_1\geq \lambda_2$ in $\{|x|<h(0)\}$. Then, $\Sigma_t$ satisfies $\lambda_1 \geq \lambda_2$ at $(x,t)\in (-h(t),h(t))\times [0,T)$, where $T$ is the singular time.    
\end{theorem}

\begin{proof}
 Towards a contradiction, we suppose that $R=\lambda_1/\lambda_2 \geq 1$ fails at some $t>0$. Then, since $R(x,t)=1$ at $|x|=h(t)$ for all $t\geq 0$, there exist some $\varepsilon \in (0,1) $, $t_0 \in (0,T)$, and $x_0 \in (-h(t_0),h(t_0))$ such that  for all $ (x,t)\in (-h(t),h(t))\times [0,t_0)$ we have
 \begin{equation}
R(x,t)>  R(x_0,t_0)=1-\varepsilon.
 \end{equation}
 Then, we know $R_t\leq 0$, $R_{xx}\geq 0$, and $R_x=0$ at $(x_0,t_0)$. This contradicts Lemma \ref{lem:ratio_evolution}, because at the minimum point $(x_0,t_0)$ we have
 \begin{equation}
0\geq  \frac{R(1-R)}{u^3Q^3}\left[R + 2|u_x|^2 (3-R)\right] =\frac{\varepsilon(1-\varepsilon)}{u^3Q^3}(1-\varepsilon+2|u_x|^2(2+\varepsilon))>0.
\end{equation}
This completes the proof.
\end{proof}
\bigskip

\begin{proposition}\label{paperclip_curvature_ratio}
  Let $\overline{\Sigma}_t\subset \mathbb{R}^3$ be the surface of revolution obtained by rotating the paperclip $\overline{\Gamma}_t$ in \eqref{paperclip identity} around the $e_1$-axis, namely
\begin{equation}
    \overline{\Sigma}_t:=\{\mathbf{x}\in \mathbb{R}^3: (\mathbf{x}^1,(|\mathbf{x}^2|^2+|\mathbf{x}^3|^2)^{\frac{1}{2}})\in \overline{\Gamma}_t \}.
\end{equation}
The principal curvatures $\lambda_1,\lambda_2$ of $\overline{\Sigma}_t$ satisfy $\lambda_1\geq \lambda_2$ for all $t<0$.
\end{proposition}

\begin{proof}
    By the identity \eqref{paperclip identity}, we know $\cos x=e^t \cosh u(x,t)$, where $u$ is the profile of $\overline{\Sigma}_t$. Therefore, by differentiating it twice we get
    \begin{equation}
e^t (u_{xx}\sinh u+|u_x|^2\cosh u)=-\cos x = -e^t \cosh u.
    \end{equation}
By algebraic manipulation, we obtain
    \begin{equation}
        [uu_{xx}+(1+|u_x|^2)]\sinh u=(1+|u_x|^2)[\sinh u- u\cosh u].
    \end{equation}
    Now, we observe that the function $f(u):=\sinh u- u\cosh u$ satisfies $f(0)=0$ and 
    \begin{equation}
     f'(u)=-u\sinh u\leq 0.   
    \end{equation}
Thus, we have $1+|u_x|^2+ u u_{xx} \leq 0$, which means $\lambda_1 \geq \lambda_2$. This completes the proof.
\end{proof}
 
\bigskip
Now, we recall the identity (5) in \cite{bourni2021collapsing}, by which $\lambda_1\geq \lambda_2$ is equivalent to $x(\lambda_2)_x\leq 0$ as described in \eqref{eq:lambda_monotone_ratio} below. By using this magical formula, we provide the main result in this section.

\begin{corollary}[{\cite[Lemma 4.4]{bourni2021collapsing}}]\label{cor:height_speed_pancake}
Let $\Sigma_t$ be the Gauss curvature flow from the initial surface $\Sigma_0:=\overline{\Sigma}_{\bar t}$  for some $\bar t<0$, where $\overline{\Sigma}_t$ is the surface of revolution given in Proposition \ref{paperclip_curvature_ratio}. Then, $\Sigma_t$ has the $Z_2\times O(2)$-symmetry, and its displacements $h(t),l(t)$ satisfy
\begin{equation}
-\frac{d}{dt}h \leq \left(\frac{2h}{l^2+h^2}\right)^2. 
\end{equation}
\end{corollary}

\begin{proof}
Since $\overline{\Sigma}_{\bar t}$ has the $Z_2\times O(2)$-symmetry, $\Sigma_t$ also enjoys the same symmetry by the uniqueness result, Proposition \ref{prop:viscosity_uniqueness}. Then, by Theorem \ref{thm:ratio_min} and Proposition \ref{paperclip_curvature_ratio}, $\Sigma_t$ satisfies $\lambda_1 \geq \lambda_2$, namely $1+|u_x|^2+uu_{xx} \leq 0$. Therefore, remembering the $Z_2$-symmetry of $u$ in \eqref{eq:profile_symmetry} we have
\begin{equation}\label{eq:lambda_monotone_ratio}
    x(\lambda_2)_x = -\frac{x u_x}{u^2Q^3}(1+|u_x|^2 + uu_{xx}) \leq 0. 
\end{equation} 
Thus, $\lambda_2(x,t)$ attains minimum at $x=\pm h(t)$. Also, observing $\lambda_1=\lambda_2$ at $\pm h e_1$, we get
\begin{equation}\label{eq:dis_speed_ineq}
    -h'=|\lambda_2(h,t)|^2\leq |\lambda_2(0,t)|^2\leq \lambda_2(0,t)\lambda_1(0,t)=-l'.
\end{equation}
Since $h,l\to 0$ as $t$ approaches the singular time $T$, integrating \eqref{eq:dis_speed_ineq} yields $l\geq h$ for all $t \in [0,T)$. Hence, we can apply Proposition \ref{prop-touching} with $z=l$ so that
 \begin{equation}
     \lambda_2(h,t) = \min \lambda_2(\cdot,t) \leq \frac{2h}{l^2+h^2}.
 \end{equation}
 Remembering $-h'=|\lambda_2(h,t)|^2$ in \eqref{eq:dis_speed_ineq}, we complete the proof.
\end{proof}

\bigskip

\begin{remark}
We note that by \eqref{eq:dis_speed_ineq} the Gauss curvature of $\Sigma_t$ attains its minimum on the rotation axis $e_1$. Although this fact will not be used in what follows, we record it here, since it is a beautiful observation in its own right.
\end{remark}

\bigskip

\section{Ancient pancake with flat sides}\label{sec 4}

In this section, we construct an ancient pancake with flat sides for the Gauss curvature flow by taking the limit of a sequence of flows $\Sigma^j_t$, which we describe below.

\bigskip

\begin{definition}[ancient pancake approximation]
We recall the surface of revolution $\overline{\Sigma}_t\subset \mathbb{R}^3$ in Proposition \ref{paperclip_curvature_ratio}, and for each $j\in \mathbb{N}$ we consider the time $T_j<0$ such that the volume of the convex body $\Omega_j$ bounded by $\overline{\Sigma}_{T_j}$ is $V(\overline{\Sigma}_{T_j}) = \text{Vol}(\Omega_t) =4\pi j$. Then, we define $\Sigma^j_t$ as the Gauss curvature flow from the initial surface $\Sigma^j_{-j}=\overline{\Sigma}_{T_j}$.    
\end{definition}

\bigskip
We observe that $\Sigma^j_t$ has the $Z_2\times O(2)$-symmetry as discussed in Corollary  \ref{cor:height_speed_pancake}. Also,  the volume of the convex body bounded by $\Sigma^j_t$ is
\begin{equation}\label{volume_approximate_flow}
    V(\Sigma^j_t)=4\pi(-t),
\end{equation}
by the definition of $T_j$ and the formula (cf. \cite{tso1985deforming})
\begin{equation}
    \frac{d}{dt}V(\Sigma^j_t)=-\int_{\Sigma^j_t} K dg=-4\pi.
\end{equation}
Hence, by \cite{andrews1999gauss} the flow exists for all $t<0$ and converges to a round point at the space-time origin.
 \bigskip

\begin{proposition}\label{prop:displacements_rough_est}
$\Sigma^j_t$ has displacements $h_j(t),l_j(t)$ given in Definition \ref{displacements} satisfying
\begin{align}
 &  2|t| \le [l_j(t)]^2h_j(t) \le 6|t|, && h_j(t) \leq \tfrac{1}{2}\pi, && l_j(t) \geq   \tfrac{2}{\sqrt \pi}|{t}|^{1/2},
\end{align}
for all $t<0$.
\end{proposition}

\begin{proof}
The cylinder $\{\mathbf{x}: |\mathbf{x}^1| \leq h_j(t), |\mathbf{x}^2|^2+|\mathbf{x}^3|^2 \leq l_j^2(t)\}$ contains $\Sigma^j_t$.  Since its volume is $2\pi [l_j(t)]^2h_j(t)$, we get $[l_j(t)]^2h_j(t) \geq 2|t|$ by \eqref{volume_approximate_flow}. Similarly, we consider the cones whose vertices are $\pm h_j(t)e_1$, sharing the base $\{\mathbf{x}^1=0,|\mathbf{x}^2|^2+|\mathbf{x}^3|^2 \leq [l_j(t)]^2\}$. Since they are enclosed by $\Sigma^j_t$, comparing the volumes, we get $[l_j(t)]^2h_j(t) \leq 6|t|$.

Finally, $h_j(t) \leq \frac{\pi}{2}$ is obvious by definition of $\Sigma^j_{-j}=\overline{\Sigma}_{T_j}$. Combining with $l_j(t)^2h_j(t) \geq 2|t|$, this yields the desired lower bound for $l_j$.
\end{proof}

\bigskip

\begin{lemma}[{\cite[Lemma 4.5 and Corollary 4.6]{bourni2021collapsing}}]\label{lem:uniform_bound_displacement}
There exists $J\geq 1$ such that for $t \leq -10^5$ and $j\geq J$, the displacements $h_j(t),l_j(t)$ of $\Sigma^j_t$ satisfy
   \begin{align}
  &h_j(t)\geq     \tfrac{99}{200}\pi, &&     l_j(t) \le 2|t|^{\frac{1}{2}}.
   \end{align}
\end{lemma}

\begin{proof}
Since Proposition \ref{prop:displacements_rough_est} implies $2|t|\leq h_j(t) [l_j(t)]^2$, combining with Corollary \ref{cor:height_speed_pancake} we have
\begin{equation}
    -h_j'(t)\leq \left(\frac{2h_j(t)}{[l_j(t)]^2+[h_j(t)]^2}\right)^2\leq  \frac{4[h_j(t)]^2}{[l_j(t)]^4} \leq \frac{[h_j(t)]^4}{t^2},
\end{equation}
for all $j \in \mathbb{N}$ and $t<0$. Therefore, for $t \in (-j,-10^5)$ we have
\begin{equation}
\frac{1}{|h_j(t)|^3}-\frac{1}{|h_j(-j)|^3}=\int_{-j}^t  \frac{d}{ds}[h_j(s)]^{-3}ds\leq  \int_{-\infty}^t \frac{3}{s^2}ds \leq \frac{3}{|t|} \leq 10^{-4}.
\end{equation}
Since $h_j(-j)\to \pi/2$ as $j\to +\infty$, we have the first inequality for $j\geq J$ and $t\leq -10^5$. Then, the second inequality follows from $[l_j(t)]^2h_j(t)\leq 6|t|$ in Proposition \ref{prop:displacements_rough_est}.
\end{proof} 

\bigskip

Now, for each $b \in (0,4]$ and $t<0$ we define an inner barrier with flat side by
\begin{equation}\label{frying_pan}
     \Phi^b_t:=\partial \{(x_1,x_2,x_3)\,:\, |x_1|\le b, \, |(x_2,x_3)|\le \phi^b(|x_1|,t)\},
\end{equation}
where
\begin{equation}
\phi^b(x,t)=\tfrac{1}{2}(-2t)^{1/2}+\sqrt{16-(x+4-b)^2}.
\end{equation}
We observe that $\Phi^b_t$ has a flat side $\{(b,x_2,x_3):  |(x_2,x_3)| \leq \tfrac{1}{2}(-2t)^{1/2} \}$, which makes $\Phi^b_t \cap \{x_1\ge0\}$ look like a \textit{frying pan}.

\bigskip

\begin{lemma}[ancient frying pan]\label{lem:frying_pan}
    $\phi^b$ is a subsolution to \eqref{eq:profile_equation}.
\end{lemma}

\begin{proof}
The graph of $\phi^b$ is an arc of a circle of radius $4$. Hence, by considering its curvature, we have $\phi^b_{xx}(1+|\phi^b_x|^2)^{-\frac{3}{2}}=-4^{-1}$. Thus, we can directly compute
\begin{equation}
    \phi^b_t=-\frac{1}{4(\phi^b-\sqrt{16-(x+4-b)^2})}<-\frac{1}{4\phi^b}=\frac{\phi^b_{xx}}{\phi^b(1+|\phi_x|^2)^{3/2}}.
\end{equation}
\end{proof}

\begin{theorem}[Theorem \ref{Main theorem (slab)}]
There exists an ancient pancake $\Sigma_t$ for the Gauss curvature flow such that it develops a singularity at the space-time origin, and $\Sigma_t$ has flat sides including $\{(\pm \frac{1}{2}\pi, r\cos\theta,r\sin\theta): r^2\leq  -\frac{1}{2}t, \theta \in S^1\}\subset \Sigma_t$ for $t\leq -100$. Moreover, $\Sigma_t$ is of class $C^{1,1}$.
\end{theorem}

 \begin{proof}
 Thanks to the bounds for $h_j(t),l_j(t)$ in Proposition \ref{prop:displacements_rough_est}, Lemma \ref{lem:uniform_bound_displacement} and the Blaschke selection theorem (together with a diagonal argument), we can choose a subsequence $n_j\to \infty$ such that \[\Sigma^{n_j}_{t_i}\to \Sigma_{t_i}  ,\]
as $n_j\to \infty$ for each fixed time $t_i:=-10^5-i$, $i=1,2,\ldots$. 

 By Lemma \ref{lemma-extra}, for each time $t\in(t_i,0)$, $\Sigma^{n_j}_{t}$ converges to $\Sigma_t$ which is the unique viscosity solution running from $\Sigma_{t_i}$. Since this holds for all $t_i$, we obtain an ancient solution to the Gauss curvature flow $\Sigma_t$ as the limit of $\Sigma^{n_j}_t$.  We note that $\Sigma_t$ has the $Z_2\times O(2)$-symmetry and satisfies the volume identity \eqref{volume_approximate_flow}. Thus, by \cite{andrews1999gauss}, it converges to a round point at the space-time origin, and it is of class $C^{1,1}$.

It remains to show that $\Sigma_t$ contains flat sides for sufficiently negative times.
Fix $b\in(0,\pi/2)$. We first show that for $j$ large, $\Phi_{-j}^b$ lies inside the initial hypersurface
$\Sigma_{-j}^j$. This is where we use the tip asymptotics of $\Sigma_{-j}^j$. Indeed, $\Sigma_{-j}^j=\overline{\Sigma}_{T_j}$ is obtained by revolving the paperclip $\overline{\Gamma}_{T_j}$ (see Proposition~\ref{paperclip_curvature_ratio}), and by the
well-known asymptotics of the paperclip, as $t\to-\infty$, the tips of $\overline{\Gamma}_t$, after re-centering around the tips, converge to the unit-speed translating grim reaper, which has width $\pi$. Meanwhile, Proposition~\ref{prop:displacements_rough_est} yields
\[
l_j(-j)-\phi^b(0,-j)\longrightarrow +\infty \qquad \text{as } j\to\infty,
\]
(where we used $\frac{2}{\sqrt\pi}>\frac{1}{\sqrt2}$). Consequently, for each fixed $b\in(0,\pi/2)$, $\Sigma_{-j}^j$ contains $\Phi_{-j}^b$ in the inside for large $j$: i.e., there exists $J_b\in\mathbb N$ such that
\[
\Phi_{-j}^b \subset \Sigma_{-j}^j \qquad \text{for all } j\ge J_b.
\]
Moreover, Proposition \ref{prop:displacements_rough_est} gives
\begin{equation}
    l_j(t) \ge \frac{2}{\sqrt{\pi}}|t|^{1/2} \ge \phi^b(0,t)\qquad \text{ for all  } t \le -100, \text{ and } j\in\mathbb N,
\end{equation}
which is because $\phi^b(0,t) \le \frac{1}{\sqrt{2}} |t|^{1/2} + 4$ and for $t \le -100$ the following holds: 
\[
\frac{2}{\sqrt\pi}|t|^{1/2}\ \ge\ \frac{1}{\sqrt2}|t|^{1/2}+4 \quad\Longleftrightarrow\quad \Bigl(\frac{2}{\sqrt\pi}-\frac{1}{\sqrt2}\Bigr)|t|^{1/2}\ge 4.
\]
Therefore, by Lemma~\ref{lem:frying_pan} together with the comparison principle,
the inclusion propagates forward in time:
\begin{equation}
    \Phi_t^b \subset \Sigma_t^j \qquad \text{for all } t\in[-j,-100], \ \text{and } j\ge J_b.
\end{equation}
Passing to the subsequence $j=n_k\to\infty$ and using the Hausdorff convergence
$\Sigma_t^{n_k}\to \Sigma_t$ for each fixed $t<0$, we obtain
\begin{equation}
    \Phi_t^b \subset \Sigma_t \qquad \text{for all } t\le -100.
\end{equation}
Finally, letting $b\to \pi/2$ and using the monotone convergence of the barriers $\Phi_t^b \nearrow \Phi_t^{\pi/2}$, we conclude that $\Sigma_t$ contains the limiting barrier region. In particular, for $t\le -100$,
\[
\Bigl\{\Bigl(\pm \frac{\pi}{2},\, r\cos\theta,\, r\sin\theta\Bigr):\ r^2\le -\tfrac12 t,\ \theta\in S^1\Bigr\}
\subset \Sigma_t,
\]
which gives the desired flat sides.
 \end{proof}

\bigskip
\bigskip
\bigskip

\section{A priori estimate for speed}
\label{sec 5}

In the following sections, we construct a compact ancient solution in a radially symmetric cylinder $B_{r_\alpha} \times \R \subset \R^3$ to the $\alpha$-GCF that has a $O(2)\times Z_2$-symmetry. For the sake of convenience, the radius of the cylinder $r_\alpha>0$ is chosen so that the unique translator asymptotic to the cylinder (\cite{urbas1988global}, \cite{urbas1999complete}) has unit speed. Namely, 
\[
r_\alpha^2 = \frac{1}{\pi}\int_{\R^2} \frac{1}{(1+|p|^2)^{2 - \frac{1}{2\alpha}}} dp = \frac{2\alpha}{2\alpha-1}.
\]
For the formula, see equation (4.5) of \cite{urbas1999complete}.

We parametrize strictly convex closed hypersurfaces by their outward normal $\nu$. To do so, we adopt the parametrization by the polar angles $\varphi \in [0,\pi]$ and $\theta\in[0,2\pi]$ as 
\[\nu = ( \sin\varphi\cos\theta,\sin\varphi\sin\theta,\cos\varphi) \in \mathbb{S}^2 .\]
Since we assume O(2)-symmetry, the parameter $\theta$ will mostly be omitted. 

\begin{lemma} [principal curvatures]
    \label{principal curvatures} For a given $C^2$ closed strictly convex hypersurface which has an $O(2)$-symmetry in the $x_1x_2$-plane, let $S(\varphi,\theta) = S(\varphi)$ be the support function. Then the principal curvatures at each point of the surface are given by
    \begin{equation}
    \lambda_1 = (S + S_{\varphi\varphi})^{-1}, \quad \lambda_2 = (S + \cot\varphi S_\varphi)^{-1}.    
    \end{equation}
        Here, $\lambda_1$ is the curvature of the curve made by the intersection of the surface and a plane containing the $x_3$-axis. 
\end{lemma}

\begin{proof} Let $i = 1$ correspond to the coordinate $\varphi$ and $i= 2$ correspond to $\theta$.    Let $\bar{\nabla}$ be the standard Riemannian connection on $(S^2, \bar{g}_{ij})$. Recall that the eigenvalues of $(\bar{\nabla}^2_{ik} S + S\bar{g}_{ik})\bar{g}^{kj} =: b_i^j$ correspond to $\lambda_i^{-1}$. Then
    \begin{equation}
        \bar g_{11} = 1, \quad \bar g_{12} = \bar g_{21} = 0, \quad \bar g_{22} = \sin^2 \varphi,
    \end{equation}
    \begin{equation}
        \bar \Gamma_{12}^2 = \bar \Gamma_{21}^2 = \cot\varphi, \quad \bar \Gamma_{22}^1 = -\sin\varphi \cos\varphi, \quad \text{all other } \bar \Gamma_{jk}^i \text{ vanish}.
    \end{equation}
    Thus $b_1^1 = (S + S_{\varphi\varphi})$, $b_2^1 = b_1^2 = 0$, and $b_2^2 = (S\sin^2\varphi + S_\varphi \sin\varphi\cos\varphi)\sin^{-2}\varphi$.
\end{proof}

\bigskip

In the next two lemmas, we prove a priori estimates on the speed $K^\alpha$ for smooth solutions showing that certain estimates on the speed are preserved under the flow. The results here are motivated by Proposition 3.2 and 3.3 of \cite{bourni2022ancient}.

\begin{lemma} [comparison with translator]
\label{barrier lemma}  Let $\Sigma_t$, for $t\in[0,T]$, be a smooth strictly convex closed solution to the $\alpha$-GCF in $\mathbb{R}^3$ which has a rotational symmetry in the $x_1x_2$-plane. If the initial data $\Sigma_0$ satisfies $K^\alpha(\varphi)\ge |\cos \varphi|$ for all $\varphi \in[0,\pi]$, then the inequality is preserved at later times. Namely, 
\begin{equation}
    K^\alpha (\varphi,t)\ge |\cos \varphi|,
\end{equation}
for $t\in [0,T]$ and $\varphi\in [0,\pi]$.
\end{lemma}

\begin{proof} First, we give a proof for $\varphi \in [\frac {\pi}{2},\pi]$. Define $u := K^\alpha(\varphi, t)+ \cos\varphi$ for $\frac{\pi}{2} \le  \varphi \le  \pi$. By Lemma \ref{principal curvatures},
    \begin{equation}
    S_t =  -K^\alpha = -(\lambda_1\lambda_2)^\alpha = -[(S+S_{\varphi\varphi})(S + \cot\varphi S_\varphi)]^{-\alpha}.
        \end{equation}
    Then 
    \begin{equation}
        \label{evolution of K^alpha}
        (K^\alpha)_t = \alpha K^{\alpha+1}[\lambda_2^{-1}(K^\alpha + (K^\alpha)_{\varphi\varphi}) + \lambda_1^{-1}(K^\alpha + \cot\varphi(K^\alpha)_\varphi)].
    \end{equation}
    Note also that $\cos\varphi$ solves the same equation as \eqref{evolution of K^alpha}, so
    \begin{equation}
    \label{parabolic PDE for u}
        u_t = \alpha K^{\alpha + 1}[\lambda_2^{-1}(u+u_{\varphi\varphi}) + \lambda_1^{-1}(u+\cot\varphi u_\varphi)].
    \end{equation}
Since $\cos\varphi \le 0$ on $[\frac{\pi}{2}, \pi]$, we have $u(\varphi, 0) \ge 0$. Then the maximum principle gives $u \ge 0$ on $[\frac{\pi}{2}, \pi] \times [0, T]$. This finishes the proof for $\varphi \in [\frac {\pi}2,\pi].$
    
In the remaining case $\varphi \in [0,\frac {\pi}2]$,  we replace the previous $u$ by $u = K^\alpha - \cos\varphi$ on $[0, \frac{\pi}{2}]$. The same argument gives the desired estimate.
\end{proof}

\bigskip

The next lemma shows a monotonicity of the Gauss curvature with respect to the angle $\varphi$ is preserved.

\begin{lemma} [monotonicity of speed in normal angle]
    \label{K^alpha monotonicity lemma} Let $\Sigma_t$, for $t\in[0,T]$, be a smooth strictly convex closed solution to the $\alpha$-GCF in $\mathbb{R}^3$ which has a rotational symmetry in the $x_1x_2$-plane and a reflection symmetry with respect to $x_3=0$, namely it has $O(2)\times Z_2$-symmetry. For the initial surface $\Sigma_0$, suppose we have $\partial_\varphi K^\alpha\le0$ for $\varphi\in[0,\frac{\pi}{2}]$. Then the inequality is preserved at later times.  \begin{equation}\label{K^alpha monotonicity} \partial_\varphi K^\alpha\le0,\end{equation} for $t\in[0,T]$ and $\varphi\in[0,\frac{\pi}{2}]$.
    
\end{lemma}

\begin{proof}
    Define $v := \partial_\varphi K^\alpha$. Differentiating \eqref{evolution of K^alpha} gives that
    \begin{align}
        \begin{split}
            \label{v_t split ver}
            v_t &= \alpha K^{\alpha +1}[\lambda_2^{-1}(v + v_{\varphi\varphi}) + \lambda_1^{-1}(v + \cot\varphi v_\varphi)] \\
            &+ (\alpha +1) K v[\lambda_2^{-1}(K^\alpha + v_\varphi) + \lambda_1^{-1}(K^\alpha + \cot\varphi v)] \\
            &+ \alpha K^{\alpha + 1}[(\lambda_2^{-1})_\varphi(K^\alpha + v_\varphi) + (\lambda_1^{-1})_\varphi (K^\alpha + \cot\varphi v)] \\
            &= \alpha K^{\alpha+1} \lambda_2^{-1} v_{\varphi\varphi} + B(\varphi,t) v_\varphi + C(\varphi,t) v.
        \end{split}
    \end{align}
    Let $A(\varphi,t) := \alpha K^{\alpha+1} \lambda_2^{-1}$. Since $\Sigma_t$ is strictly convex, $A(\varphi,t) > 0$, making the equation strictly parabolic. Note that $C(\varphi,t)$ is bounded on the compact domain $[0, \frac{\pi}{2}] \times [0,T]$. 
    
    To apply the standard maximum principle, we need to handle the zeroth-order term $C(\varphi,t)v$ whose sign is not strictly non-positive. Let $C_0 := \sup_{[0, \frac{\pi}{2}] \times [0,T]} |C(\varphi,t)|$ and define a new function 
    $$ \tilde{v}(\varphi,t) := e^{-\beta t} v(\varphi,t) $$
    for a constant $\beta > C_0$. Direct computation shows that $\tilde{v}$ satisfies
    \begin{equation}
        \label{v_tilde eq}
        \tilde{v}_t = A(\varphi,t) \tilde{v}_{\varphi\varphi} + B(\varphi,t) \tilde{v}_\varphi + (C(\varphi,t) - \beta) \tilde{v}.
    \end{equation}
    Now the coefficient of the zeroth-order term satisfies $C(\varphi,t) - \beta < 0$, which puts the equation in the standard form for the weak maximum principle. 
    
    Observe that, by the symmetry of the solution, the boundary conditions are $\tilde{v}(0,t) = \tilde{v}(\frac{\pi}{2},t) = 0$ for all $t \in [0,T]$. Furthermore, the initial condition is $\tilde{v}(\varphi,0) = v(\varphi,0) \le 0$ for $\varphi \in[0, \frac{\pi}{2}]$.
    
    Suppose for a contradiction that $\tilde{v}$ attains a strictly positive maximum over $[0, \frac{\pi}{2}] \times [0,T]$ at some interior point $(\varphi_0, t_0)$ with $t_0 > 0$. At this maximum point, we must have $\tilde{v} > 0$, $\tilde{v}_t \ge 0$, $\tilde{v}_\varphi = 0$, and $\tilde{v}_{\varphi\varphi} \le 0$. Substituting these into \eqref{v_tilde eq} yields
    $$ 0 \le \tilde{v}_t = A(\varphi_0,t_0) \tilde{v}_{\varphi\varphi} + (C(\varphi_0,t_0) - \beta)\tilde{v} < 0, $$
    which is a contradiction. Thus, $\tilde{v}(\varphi,t) \le 0$ everywhere. Since $v(\varphi,t) = e^{\beta t} \tilde{v}(\varphi,t)$ and $e^{\beta t} > 0$, we conclude that $v(\varphi,t) \le 0$ for all $t \in[0,T]$ and $\varphi \in [0, \frac{\pi}{2}]$.
\end{proof}

\bigskip

\section{Ancient sausage for flows by power of Gauss curvature} \label{sec 6}
\subsection{Construction for small power}\label{subsec 6.1}

Here we prove the existence of an \textit{ancient sausage}, namely an ancient solution asymptotic to a round cylinder, for $\alpha \in (1/2,1)$. A slight modification of the argument presented here also works for the case $\alpha = 1$, but we restrict attention to $\alpha \in (1/2,1)$, since the case $\alpha=1$ was already treated in \cite{choi2024uniqueness} in greater generality. Unless otherwise mentioned, we assume $\alpha \in (1/2,1)$ in the assertions and proofs in this subsection.

\medskip

Consider the sequence of convex closed smooth hypersurfaces $\Sigma^i$ which have $O(2)\times Z_2$-symmetry in $\mathbb{R}^3$ and 
\[ K^\alpha(\varphi ) = (\cos^2\varphi +i^{-2})^{1/2} .  \] By \cite{nirenberg1953weyl}, there exists such a $\Sigma^i$, and it is unique modulo translations in $\mathbb{R}^3$. Here we impose $O(2)\times Z_2$-symmetry with respect to coordinate axes in $\mathbb{R}^3$ and hence such a solution is unique. For each $\Sigma^i$, there exists a unique smooth $\alpha$-GCF for time $t\in[0,T_i)$. Here $T_i$ is the extinction time where the solution converges to a round point. (See \cite{brendle2017asymptotic}.) After a time translation, we denote this solution by 
\[\Sigma^i_t , \quad \text{for } t\in [-T_i,0).\] It follows that $\Sigma^i_t $ preserves $O(2)\times Z_2$-symmetry.

\begin{lemma} \label{lemma-cigarproperty} For $\alpha\in(1/2,1)$, the $\alpha$-Gauss curvature flows $\Sigma^i_t$, for $t\in[-T_i,0)$, satisfy the following properties: 

\begin{enumerate}[(1)]
\item  $K^\alpha(\varphi,t) \ge 
\cos\varphi$, for $\varphi\in [0,\frac{\pi}{2}] $, $t\in[-T_i,0)$.
\item   $\partial_\varphi K^\alpha (\varphi,t)  <0$ for $\varphi\in (0,\frac{\pi}{2}) $, $t\in[-T_i,0)$. 
\item $T_i\to \infty $ as $i\to \infty$.
\item The initial displacement $h_i(-T_i)$ is nondecreasing and $\lim_{i\to \infty} h_i(-T_i) = r_\alpha $.
\item The displacement $l_i(t)$ satisfies $|l_i(t)|\ge |t|$. 
\end{enumerate}

\begin{proof}

(1): Let $u = K^\alpha - \cos\varphi$, whose initial data at $t=-T_i$ is given by
\begin{equation}
u(\varphi, -T_i) = (i^{-2} + \cos^2\varphi)^{\frac{1}{2}} - \cos\varphi \ge (\cos^2\varphi)^{\frac{1}{2}} - \cos\varphi = 0.    
\end{equation}
Then Lemma \ref{barrier lemma} gives the result.

(2): Define $u = \partial_{ \varphi} K^\alpha$. Then
\begin{equation}
u(\varphi, -T_i) = -\frac{\sin(2\varphi)}{2(i^{-2} + \cos^2\varphi)^{\frac{1}{2}}} \le 0.
    \end{equation}
Now Lemma \ref{K^alpha monotonicity lemma} yields the result.

(4): Recall that Lemma \ref{principal curvatures} gives
\begin{equation}
K = \lambda_1 \lambda_2 = \frac{1}{(S+S_{\varphi\varphi})(S + \cot\varphi S_\varphi)}.
\end{equation}
Denote $A = S+S_{\varphi\varphi}$, $B = S + \cot\varphi S_\varphi$ so that $K = (AB)^{-1}$.
Observe that
\begin{equation}
B\sin\varphi  = S \sin\varphi + S_\varphi \cos\varphi.    
\end{equation}
Differentiating this, we have
\begin{equation}
\frac{d}{d\varphi}(B\sin\varphi ) = A \cos\varphi,    
\end{equation}
so
\begin{equation}
\label{derivative of sin phi B}
\frac{d}{d\varphi}\left(\frac{1}{2}(B\sin\varphi )^2\right) = AB \sin\varphi \cos\varphi = \frac{\sin\varphi \cos\varphi}{K}.
\end{equation}

Observe that  at the equator $\varphi = \frac{\pi}{2}$, $ B \sin\varphi  = S(\frac{\pi}{2}) = h$. 
Moreover, as $\varphi \to 0^+$, $B \sin\varphi  \to 0$ because $\sin\varphi \to 0$ and $S_\varphi(\varphi)\to S_{\varphi}(0)=0$ by the symmetry. Integrating \eqref{derivative of sin phi B} from 0 to $\frac{\pi}{2}$, we have
\begin{equation}
\label{h formula}
    h^2 = 2 \int_0^{\frac{\pi}{2}} \frac{\sin\varphi \cos\varphi}{K} d\varphi.
\end{equation}
Using $K^\alpha = (\cos^2\varphi +i^{-2})^{1/2}$ and substituting $y = \cos^2 \varphi$, we have
\begin{equation}
            h_i(-T_i)^2 = 2 \int_0^{\frac{\pi}{2}} \frac{\sin\varphi \cos\varphi}{(\cos^2\varphi +i^{-2})^{1/2\alpha}} d\varphi = \int_0^1 (y + i^{-2})^{-\frac{1}{2\alpha}} dy,
\end{equation}
where $h_i$ is the displacement $h(t)$ of $\Sigma^i_t$. Then, we observe that the integrand decreases as $i$ gets larger. Moreover,
\begin{equation}
\lim_{i \to \infty} h_i(-T_i)^2 = \int_0^1 y^{-\frac{1}{2\alpha}} dy = \frac{2\alpha}{2\alpha - 1} = r_\alpha^2.    
\end{equation}
Hence, we get $h_i(-T_i) \nearrow r_\alpha$.

(3): Observe that the standard reconstruction formula $X=S\nu+\bar\nabla S$ yields the meridian
coordinates $(r(\varphi),z(\varphi))$, where $r=\sqrt{x_1^2+x_2^2}$ and $z=x_3$, as
\begin{equation}\label{eq:rz-from-support}
r(\varphi)=S(\varphi)\sin\varphi+S_\varphi(\varphi)\cos\varphi,\qquad
z(\varphi)=S(\varphi)\cos\varphi-S_\varphi(\varphi)\sin\varphi.
\end{equation}
Moreover, by Lemma~\ref{principal curvatures} we have
\[
\lambda_1(\varphi)=\bigl(S+S_{\varphi\varphi}\bigr)^{-1},\qquad
\lambda_2(\varphi)=\bigl(S+\cot\varphi\,S_\varphi\bigr)^{-1}.
\]
Differentiating \eqref{eq:rz-from-support} and using $\lambda_1^{-1}=S+S_{\varphi\varphi}$ gives
\begin{equation}\label{eq:drdz}
r_\varphi=\lambda_1^{-1}\cos\varphi,\qquad z_\varphi=-\lambda_1^{-1}\sin\varphi.
\end{equation}
Also, since $r=(S+\cot\varphi S_\varphi)\sin\varphi=\lambda_2^{-1}\sin\varphi$, we obtain
\begin{equation}\label{eq:lambda2}
\lambda_2(\varphi)=\frac{\sin\varphi}{r(\varphi)}.
\end{equation}
Combining \eqref{eq:drdz}--\eqref{eq:lambda2} with $K=\lambda_1\lambda_2$ yields
\begin{equation}\label{eq:zphi}
-z_\varphi(\varphi)=\lambda_1^{-1}\sin\varphi=\frac{\sin^2\varphi}{r(\varphi)\,K(\varphi)}.
\end{equation}

Let $h_i(-T_i)$ and $l_i(-T_i)$ be the displacements of $\Sigma^i_{-T_i}$ as in
Definition~\ref{displacements}. By symmetry,
\[
l_i(-T_i)=z_i(0)-z_i\Bigl(\frac{\pi}{2}\Bigr)=\int_0^{\pi/2}\bigl(-z_{i,\varphi}(\varphi)\bigr)\,d\varphi.
\]
Using \eqref{eq:zphi} we get
\begin{equation}\label{eq:l-integral}
l_i(-T_i)=\int_0^{\pi/2}\frac{\sin^2\varphi}{r_i(\varphi)\,K_i(\varphi,-T_i)}\,d\varphi.
\end{equation}
Since the surface is convex and rotationally symmetric, $r_i(\varphi)\le r_i(\pi/2)=h_i(-T_i)$ for
$\varphi\in[0,\pi/2]$. Hence
\begin{equation}\label{eq:l-lower1}
l_i(-T_i)\ge \frac{1}{h_i(-T_i)}\int_0^{\pi/2}\frac{\sin^2\varphi}{K_i(\varphi,-T_i)}\,d\varphi.
\end{equation}
At $t=-T_i$ we have the prescribed curvature
\[
K_i^\alpha(\varphi,-T_i)=\bigl(\cos^2\varphi+i^{-2}\bigr)^{1/2},
\qquad\text{so}\qquad
K_i(\varphi,-T_i)=\bigl(\cos^2\varphi+i^{-2}\bigr)^{\frac{1}{2\alpha}}.
\]
Therefore
\[
\int_0^{\pi/2}\frac{\sin^2\varphi}{K_i(\varphi,-T_i)}\,d\varphi
=
\int_0^{\pi/2}\sin^2\varphi\bigl(\cos^2\varphi+i^{-2}\bigr)^{-\frac{1}{2\alpha}}\,d\varphi.
\]
With the change of variables $y=\cos\varphi$ this becomes
\[
\int_0^1 \frac{\sqrt{1-y^2}}{(y^2+i^{-2})^{\frac{1}{2\alpha}}}\,dy
\ge
\frac{1}{\sqrt2}\int_0^{1/\sqrt2}(y^2+i^{-2})^{-\frac{1}{2\alpha}}\,dy.
\]
Letting $u=iy$ yields
\[
\int_0^{1/\sqrt2}(y^2+i^{-2})^{-\frac{1}{2\alpha}}\,dy
=
i^{\frac{1}{\alpha}-1}\int_0^{i/\sqrt2}(1+u^2)^{-\frac{1}{2\alpha}}\,du
\ge
c_\alpha\, i^{\frac{1}{\alpha}-1},
\]
where $c_\alpha:=\int_0^{1}(1+u^2)^{-\frac{1}{2\alpha}}\,du>0$ depends only on $\alpha$.
Combining with \eqref{eq:l-lower1}, we obtain
\begin{equation}\label{eq:l-lower-final}
l_i(-T_i)\ge \frac{c_\alpha}{h_i(-T_i)}\, i^{\frac{1}{\alpha}-1}.
\end{equation}

Let $V_i(t)$ be the volume enclosed by $\Sigma^i_t$. By convexity and $O(2)\times Z_2$-symmetry,
$\Sigma^i_{-T_i}$ encloses the double cone with base radius $h_i(-T_i)$ and height $2l_i(-T_i)$.
Hence
\begin{equation}\label{eq:V-lower}
V_i(-T_i)\ge \frac{2\pi}{3}\,h_i(-T_i)^2\,l_i(-T_i).
\end{equation}
Substituting \eqref{eq:l-lower-final} into \eqref{eq:V-lower} gives
\[
V_i(-T_i)\ge C_\alpha\, h_i(-T_i)\, i^{\frac{1}{\alpha}-1}.
\]
By Lemma~6.1 (4) we have $h_i(-T_i)\nearrow r_\alpha>0$, so for all large $i$,
$h_i(-T_i)\ge r_\alpha/2$. Therefore,
\begin{equation}\label{eq:V-growth}
V_i(-T_i)\ge C'_\alpha\, i^{\frac{1}{\alpha}-1}\xrightarrow[i\to\infty]{}\infty,
\end{equation}
since $\alpha\in(1/2,1)$ implies $\frac{1}{\alpha}-1>0$.

For the $\alpha$-Gauss curvature flow,
\[
-\frac{d}{dt}V_i(t)=\int_{\Sigma^i_t}K_i^\alpha\,d\mu=\int_{\mathbb S^2}K_i^{\alpha-1}\,d\omega.
\]
At $t=-T_i$,
\[
K_i^{\alpha-1}(\varphi,-T_i)
=
\bigl(\cos^2\varphi+i^{-2}\bigr)^{\frac{\alpha-1}{2\alpha}}
=
\bigl(\cos^2\varphi+i^{-2}\bigr)^{-\beta},
\qquad
\beta:=\frac{1-\alpha}{2\alpha}\in\Bigl(0,\frac12\Bigr).
\]
Using rotational symmetry and $y=\cos\varphi$ we get
\[
-\frac{d}{dt}V_i(-T_i)
=
4\pi\int_0^1 (y^2+i^{-2})^{-\beta}\,dy
\le
4\pi\int_0^1 y^{-2\beta}\,dy
=
\frac{4\pi}{1-2\beta}
=:C_\alpha.
\]
Thus $-\frac{d}{dt}V_i(-T_i)\le C_\alpha$ uniformly in $i$.

Since $V_i(0)=0$ and $V_i(-T_i)>0$, we have
\[
V_i(-T_i)=\int_{-T_i}^{0}\left(-\frac{d}{dt}V_i(t)\right)dt
\le \int_{-T_i}^{0} C_\alpha\,dt
= C_\alpha T_i.
\]
Hence
\[
T_i\ge \frac{V_i(-T_i)}{C_\alpha}\xrightarrow[i\to\infty]{}\infty,
\]
where we used \eqref{eq:V-growth}.

(5): By (1), at the north pole $\varphi = 0$ we have
    \begin{equation}
     \label{-l'(t)}
     -l'(t) = K^\alpha(0,t) \ge \cos 0 = 1.    
    \end{equation}
     Integrating from $t$ to $0$ gives 
    \begin{equation}
     l(t) \ge -t        
    \end{equation}
    for $t \in [-T_i, 0]$.
\end{proof}

\end{lemma}

\bigskip

\begin{proposition}\label{prop-volupperbd} 
Let $V_i(t)$ be the volume of the convex body bounded by $\Sigma_t^i$. An upper bound of $V_i(t)$ is given by
\begin{equation}
 V_i(t) \le 2\pi r_\alpha^2(-t).   
\end{equation}

\begin{proof}
By Lemma \ref{lemma-cigarproperty} (1) and $Z_2$ symmetry, we have $K^\alpha(\varphi,t) \ge |\cos \varphi|$.    Using rotational symmetry, we have
    \begin{align}
        \begin{split}
            -\frac{dV_i}{dt} &= \int K^{\alpha} d\mu = \int_{S^2} K^{\alpha-1} d\omega \\
            &= 2\pi\int_{0}^{\pi} |\sin\varphi| \cdot K^{\alpha-1}(\varphi) d\varphi \le 4\pi\int_0^{\frac{\pi}{2}}\sin\varphi (\cos\varphi)^{1-\frac{1}{\alpha}} d\varphi.
        \end{split}
    \end{align} In the last inequality, we used $\alpha\le 1$. Note that
    \begin{equation}
    \label{r_alpha}
    2\int_0^{\frac{\pi}{2}} \sin\varphi (\cos\varphi)^{1-\frac{1}{\alpha}} d\varphi = \frac{2\alpha}{2\alpha-1} = r_\alpha^2,
    \end{equation}
    so we have
    \[
    -\frac{dV_i}{dt} \le 2\pi r_\alpha^2. 
    \] The proposition follows by integrating it from $t$ to $0$. 
\end{proof}
\end{proposition}

\bigskip

\begin{proposition}\label{prop-h_ilowerbd}
     There exists $t_\alpha>-\infty$
 such that for $t\le t_\alpha$, 
 \begin{equation}
 K_i( \tfrac{\pi}{2}, t) \le 
\frac{C_\alpha }{|l_i(t)|^2 }\le C_\alpha |t|^{-2}.
 \end{equation}
 This, in particular, implies that for $t\in [-T_i, t_\alpha]$, 
 \begin{equation}
   h_i(t) \ge h_i(-T_i) - \int_{-T_i}^{t} C_\alpha^\alpha |s|^{-2\alpha }ds\ge h_i(-T_i)- \frac{C_\alpha ^\alpha}{2\alpha-1} |t|^{1-2\alpha}.  
 \end{equation}
\end{proposition}

\begin{proof}
    The second assertion follows by integrating the first assertion. The proof of the first assertion follows by a variant of the argument which appeared in \cite[Lemma 4.4]{bourni2021collapsing}. We work on the slice of $\Sigma^i_t$  with the $x_1x_3$-plane. Let us denote this $Z_2^2$-symmetric curve by $\Gamma^i_t :=\{(x_1,x_3)\,:\, (x_1,0,x_3) \in \Sigma^i_t\}$.

Since $l_i(t)\ge |t|$ and $h_i(t)\le r_\alpha $, there is $t_\alpha>-\infty$ such that $l_i(t)/2 \ge h_i(t)$ for all $t\le t_\alpha$. For such a $t\le t_\alpha$, let us apply Proposition \ref{prop-touching} for $z=\frac{l_i(t)}2$. Thus, there exists $\mathbf{x}\in \Gamma^i_t$, such that $|\langle \mathbf{x}, e_3\rangle | \le l_i(t)/2$ and the curvature of $\Gamma^i_t$ at $\mathbf{x}$, namely $\lambda_1(\mathbf{x})$,  satisfies 
\begin{equation}
    \lambda_1(\mathbf{x})\le \frac{8h_i(t)}{l_i^2(t)}.
\end{equation}
Moreover, by convexity, $\langle \mathbf{x},e_1\rangle \ge h_i(t)/2$ as otherwise the curve $\Gamma^i_t$ should be included in $\{|x_3|< l_i(t)\}$, a contradiction. This implies 
\begin{equation}
\lambda_2 (\mathbf{x})=\frac{\langle \nu,e_1\rangle }{\langle \mathbf{x}, e_1\rangle}\le  \frac{2}{h_i(t)}.    
\end{equation}
Now, the assertion follows since 
\begin{equation}
    K(\frac{\pi}{2},t)  \le \lambda_1(\mathbf{x})\lambda_2(\mathbf{x}) \le \frac{16}{ l_i^{2}(t)},
\end{equation}
and $l_i(t)\ge |t|$.   
\end{proof}

\bigskip

 We are now in a position to prove the main existence theorem for $\alpha \in (1/2,1)$. In addition, we establish some further properties of the solution beyond those stated in Theorem \ref{Main theorem (cylinder)}. 
 
 \begin{theorem}[Theorem \ref{Main theorem (cylinder)} for $\alpha <1$] 
 In the round closed cylinder $\overline{B_{r_\alpha}} \times \R \subset \R^3$, there exists an $O(2)\times Z_2$-symmetric, compact, ancient solution $\{\Sigma_t\}_{t \in (-\infty, 0)}$ to the $\alpha$-GCF (in the viscosity sense) for $\frac{1}{2} < \alpha < 1$, which does not lie in any smaller cylinder. There is a sufficiently negative time $t_0>-\infty$ such that the solution becomes smooth and strictly convex for later time $t\ge t_0$. Moreover, at time $t=0$, the solution shrinks at the origin as a round point.    
\end{theorem}
 
\begin{proof}
Consider the $O(2)\times Z_2$-symmetric smooth $\alpha$-GCF $\Sigma^i_t$, for $t\in [-T_i,0)$, constructed at the beginning of this section. In view of the upper bound on the volume (Proposition \ref{prop-volupperbd}), the lower bound on the displacement $h_i(t)$ (Proposition \ref{prop-h_ilowerbd}), and convexity, we obtain a uniform upper bound on $l_i(t)$ for large $i$ when $|t|$ is large. More precisely, there exist $t_\alpha$ and $C_\alpha$ such that if $t\le t_\alpha$ then $l_i(t) \le C_\alpha |t|$ for all sufficiently large $i\ge i_0(t)$. Note also that there is a lower bound on $l_i(t)$ in Lemma \ref{lemma-cigarproperty} (5), an upper bound $h_i(t)\le r_\alpha$ (Lemma \ref{lemma-cigarproperty} (4)), and a lower bound on $h_i(t)$ (Proposition \ref{prop-h_ilowerbd}), all of which are uniform for large $i$. 

Choose a sequence of nonincreasing times $t_i\to -\infty$. Thanks to the displacement bounds mentioned above and the Blaschke selection theorem (together with a diagonal argument), we can extract a subsequence $n_j$ such that $\Sigma^{n_j}_{t_i}$ converges to a convex closed hypersurface $\Sigma_{t_i}$ as $n_j\to \infty$ in the Hausdorff distance, for all $t_i$. In view of Lemma \ref{lemma-extra} (compactness), this implies that, for all $t\in(-\infty,0)$, $\Sigma^{n_j}_t$ converges to $\Sigma_t$, and $\Sigma_t$ for $t\in(-\infty,0)$ is an ancient viscosity solution to the $\alpha$-GCF. Like $\Sigma^i_t$, the limit $\Sigma_t$ lies inside the cylinder $\overline{ B_{r_\alpha}}\times \mathbb{R}$. Proposition \ref{prop-h_ilowerbd} implies that the limit solution $\Sigma_t$ satisfies the displacement bound $h(t)\ge r_{\alpha} - C_\alpha' |t|^{1-2\alpha}$ for $t\le t_\alpha$. In particular, this shows that $\Sigma_t$ does not lie inside any smaller cylinder. Since $l(t)\to \infty$ and $h(t)\to r_\alpha$ as $t\to -\infty$, convexity implies that the solution sweeps out the interior of the cylinder as $t\to -\infty$.

After obtaining an ancient viscosity solution $\Sigma_t$, we are left with the remaining assertions that $\Sigma_t$ become strictly convex and smooth after certain negative time $t\ge t_0>-\infty$ and shrinks to a round point at the space-time origin. (Note the fact that $\Sigma^i_t$ shrinks to a round point at the origin does not necessarily mean that $\Sigma_t$ does have the same asymptotics.) Indeed, it suffices to prove the first assertion: if we show that the solution $\Sigma_t$ becomes smooth and strictly convex from some negative time, then, since the solution is centrally symmetric, \cite{andrews2016flow} implies the convergence to a round point.\footnote{A nontrivial thing here is to explain why the limit $\Sigma_t$ extincts at $t=0$. If it extincts at earlier or later time (respectively), then an argument using Lemma \ref{lemma-extra} and a barrier (comparing with a shrinking sphere) implies that $\Sigma^i_t$ should exist at earlier, or later times (respectively) for large $i$, a contradiction. See the last paragraph of the current proof for a similar argument.} In \cite[Proposition 2.8]{choi2024uniqueness}, it was shown that if a point on the solution is away from the initial surface, then the solution is locally smooth and strictly convex. From this observation, it suffices to prove that $h(t)$ is strictly less than $r_\alpha$ for some negative time $t\ge t_0>-\infty$. Indeed, this would imply that $\Sigma_{t_0}$ is strictly contained inside $\Sigma_{t_1}$ for sufficiently negative $t_1\ll t_0$, and thus $\Sigma_t$ is smooth for $t\ge t_0$. 

Suppose instead that $h(t)=r_\alpha$ for all $t<0$. In view of $l(t)\ge h(t)$, this implies that, for all $t<0$, $\Sigma_t$ converges to a compact convex surface bounding a convex body with nonempty interior which is independent in $t$. In particular, this would imply that $\Sigma^{n_j}_t$ contains a nonempty open set (uniform in $t$) in its interior for large $n_j\ge n_{j_0(t)}$ for all $t<0$. Using this and a barrier argument, we can deduce that $\Sigma^{n_j}_t$ for large $n_j$ cannot shrink to a point at $t=0$, a contradiction. This proves that $\Sigma_t$ becomes smooth and converges to a round point at the origin as $t\to 0^-$.
\end{proof}

\bigskip

\subsection{Construction for large power}\label{subsec 6.2}

We construct an ancient sausage for large $\alpha>1$, which is relatively straightforward. In this range, the rotationally symmetric translators touch the boundary of the cylinder and we can directly glue two translators to make a compact ancient solution. 

\medskip 

Let $\Sigma^2 \subset \mathbb{R}^3$ be the translator which is asymptotic to $B_{r_{\alpha}}\times \mathbb{R}$, moves in the $+e_3$-direction and $0\in \Sigma$. According to Urbas \cite{urbas1988global, urbas1999complete}, there is a radially symmetric convex function $u$ on $B_{r_{\alpha}}$ such that $u(0)=0$, $\nabla u(B_{r_{\alpha}})= \mathbb{R}^2$ and  \[\Sigma = \partial \{ (x_1,x_2,x_3)\in \mathbb{R}^3 \,:\, x_3>  u(x_1,x_2) \} .\]
By a barrier argument, in \cite{urbas1999complete}, it was also shown that the graph of $u$ is not complete in the sense that $\lim_{|x|\to r_{\alpha}}u(x)<\infty$.  Here, since we are in the rotationally symmetric case, we prove it by a direct computation.

\begin{lemma}
\label{alpha > 1 flat side}
    For the radially symmetric convex graphical function $u : B_{r_\alpha} \to \R$ of the translator such that $u(0)=0$, we have
    \[
    \lim_{x \to \partial B_{r_\alpha}} u(x)=:M < \infty, 
    \]
    In particular, the translator is given by 
    \[\partial \{ (x_1,x_2,x_3)\in \mathbb{R}^3\,:\, x_3\ge u(x_1,x_2),\, (x_1,x_2)\in B_{r_\alpha}\} =  \mathrm{graph}_{B_{r_\alpha}}u \, \cup\, ( \partial B_{r_\alpha}\times [M,\infty)).\]
\end{lemma}

\begin{proof}
We consider $u = u(r)$ as a radial function. Then the translator equation is given by
   \begin{equation}
    1=\frac{(u''u')^\alpha}{r^\alpha(1+|u'|^2)^{2\alpha-\frac{1}{2}}}.
\end{equation}
Take $\tilde{r} > 0$ so that $|u'| \ge 1$ for all $r \ge  \tilde{r}$. Since $|u'|^2 \le 1 + |u'|^2$, for $r \ge \tilde{r}$ we have
\[
2^{\frac{1}{2} - 2\alpha} |u'|^{1 - 4\alpha} \le (1 + |u'|^2)^{\frac{1}{2} - 2\alpha} \le |u'|^{1 -4\alpha}.
\]
Thus for $r \ge \tilde{r}$
\[
2^{\frac{1}{2} - 2\alpha} r_\alpha^{-\alpha} (u'')^\alpha (u')^{1-3\alpha} \le 1 \le \tilde{r}^{-\alpha} (u'')^\alpha (u')^{1-3\alpha}. 
\]
Then
\[
u'' (u') ^{\frac{1 - 3\alpha}{\alpha}} = \frac{\alpha}{1-2\alpha} ((u')^{\frac{1 -2\alpha}{\alpha}})' \le 2^{2 - \frac{1}{2\alpha}} r_\alpha,
\]
and equivalently
\[
((u')^{\frac{1 -2\alpha}{\alpha}})' \le -C_\alpha < 0
\]
where $C_\alpha > 0$ is a positive constant depending on $\alpha$. Integrating from $r$ to $r_\alpha$, we have
\begin{equation}
   (u')^{\frac{1-2\alpha}{\alpha}} \ge (u'(r_\alpha))^{\frac{1-2\alpha}{\alpha}} + C_\alpha(r_\alpha - r) \ge C_\alpha(r_\alpha - r). 
\end{equation}
Applying the power, we finally get
\begin{equation}
    u' \le C_\alpha(r_\alpha - r)^{\frac{\alpha}{1-2\alpha}}.
\end{equation}
When $\alpha > 1$, $-1 < \frac{\alpha}{1-2\alpha} < 0$, so the right-hand side is integrable. Thus,
\begin{equation}
    \int_r^{r_\alpha} |u'| dr<+\infty.
\end{equation}
\end{proof}

\bigskip

Now we build the ancient sausage by attaching two translators with the cylinder:

\begin{figure}[htbp]
    \centering
    \includegraphics[width=0.8\linewidth]{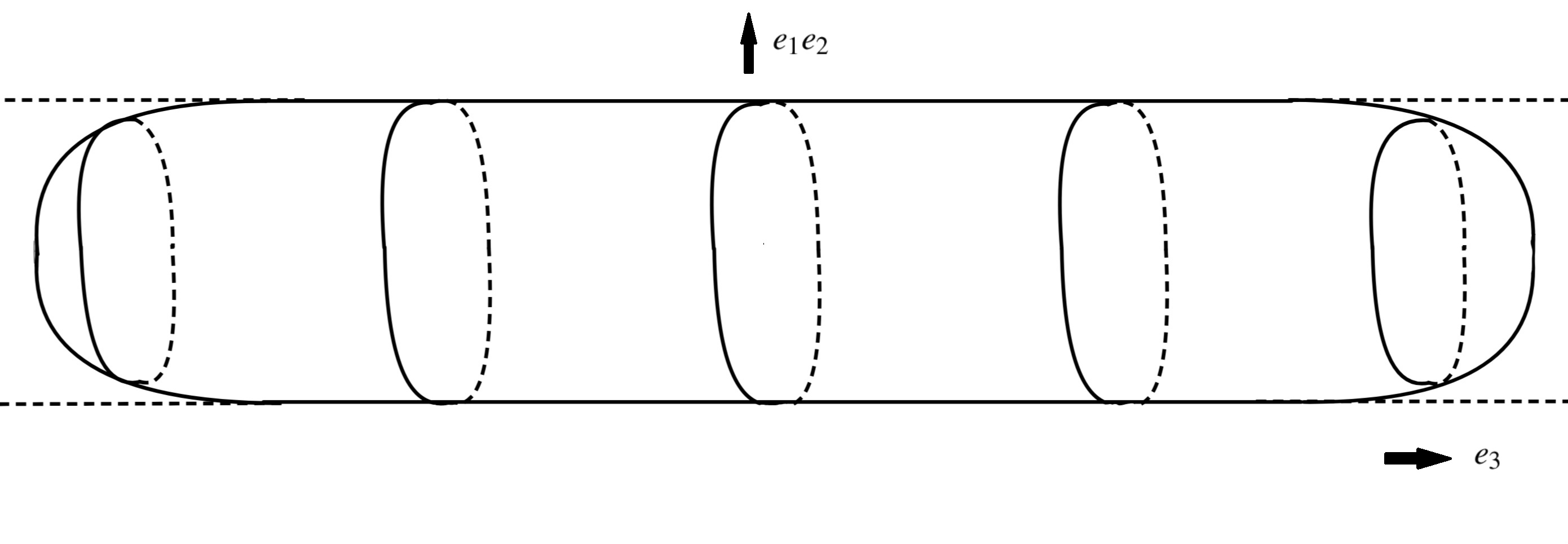}
    \caption{Ancient sausage for $\alpha > 1$}
    \label{ancient sausage for alpha > 1}
\end{figure}

\begin{theorem}[Theorem \ref{Main theorem (cylinder)} for $\alpha>1$]    In the round closed cylinder $\overline{B}_{r_\alpha} \times \R \subset \R^3$, there exists a $O(2)\times Z_2$-symmetric, compact, convex, ancient solution $\{\Sigma_t\}_{t \in (-\infty, T)}$ to the $\alpha$-GCF for $\alpha>1$. This solution does not lie in any smaller cylinder. Additionally, the solution satisfies the following:
    \begin{enumerate}[\, (1)]
\item For $t\le 0$, $\Sigma_t$ is obtained by gluing two translators coming from two opposite directions: 
\begin{equation*} \begin{aligned} \Sigma_t &= \partial \{(x_1,x_2,x_3)\,:\, u(x_1,x_2)-M+t< x_3 < -u(x_1,x_2)+M-t   \}  \\
&= \mathrm{graph}_{B_{r_\alpha}} (u-M+t)\,  \cup\, (  \partial B_{r_\alpha} \times [t,-t] )\, \cup\,  \mathrm{graph}_{B_{r_\alpha}} (-u+M-t) .\end{aligned}\end{equation*}
Here, $u:B_{r_{\alpha}}\to \mathbb{R}$ is the convex radially symmetric graph function representing the translator with minimum height $\min u =u(0)=0$, and $M=\lim_{x\to \partial B_{r_\alpha}} u(x)<\infty$.

\item For $t>0$, $\Sigma_t$ becomes smooth strictly convex and converges to a round point at the origin as $t\to T$. The extinction time $T>0$ depends only on $\alpha$ and it is strictly less than $M$.
      \end{enumerate}

\end{theorem}

\begin{proof} Consider a family of convex surfaces $\{\bar \Sigma_t\}_{t\in(-\infty,M)}$, where $\bar \Sigma_t$ is the boundary of the convex region enclosed by two translators coming from opposite directions:
\[
\bar \Sigma_t := \partial \{ (x_1,x_2,x_3) : u(x_1,x_2)-M+t < x_3 < -u(x_1,x_2)+M-t\}.
\]
Since the enclosed region is the intersection of the regions enclosed by the two translators, $\bar \Sigma_t$ is a (viscosity) subsolution to the $\alpha$-GCF for all $t\in(-\infty,M)$. Moreover, $\bar \Sigma_t$ also serves as a viscosity supersolution for $t\in(-\infty,0)$. Indeed, a smooth strictly convex solution can only touch $\bar \Sigma_t$ from the outside at points away from the cylinder, where $\bar \Sigma_t$ is itself a smooth solution (being part of a translator.)

Define $\Sigma_t$ by setting $\Sigma_t = \bar \Sigma_t$ for $t\le -1$, and for $t\ge -1$, let $\Sigma_t$ denote the unique (viscosity) solution to the $\alpha$-GCF running from $\Sigma_{-1}$ at time $t=-1$. By uniqueness and the preceding observation, $\Sigma_t = \bar \Sigma_t$ for all $t\le 0$. For $t>0$, since $\bar \Sigma_t$ is a supersolution, $\Sigma_t$ must lie inside the region enclosed by $\bar \Sigma_t$. As $\bar \Sigma_t$ lies strictly inside the cylinder for $t>0$, it follows that $\Sigma_t$ also lies strictly inside. Comparing $\Sigma_t$ with $\Sigma_{t'}$ for some $t'\ll t$, one sees that every point of $\Sigma_t$ has moved inward relative to $\Sigma_{t'}$. By \cite[Proposition 2.8]{choi2024uniqueness}, the solution therefore becomes smooth and strictly convex for $t>0$. Since $\Sigma_t$ is centrally symmetric, \cite{andrews2016flow} implies that $\Sigma_t$ converges to a round point as $t\to T=T(\alpha)$. Finally, because $\bar \Sigma_t$ is a subsolution that vanishes at $t=M$, it follows that $T<M$.

\end{proof}

\section*{Acknowledgments}
The authors were supported by the National Research Foundation(NRF) grant funded by the Korea government(MSIT) (RS-2023-00219980). BC has been partially supported by NRF of Korea grant No. 2022R1C1C1013511. KC has been partially supported by KIAS Individual Grant MG078902 and RS2024-00345403 funded by the Korea government (MSIT). 

We would like to thank Theodora Bourni and Mat Langford for their interest and insightful comments.

\bibliographystyle{alpha}
\bibliography{bib} 

\end{document}